\documentclass[aos,preprint]{imsart}

\usepackage{amsmath, amsthm, amssymb}
\usepackage{graphicx}
\usepackage{verbatim}
\usepackage{caption}
\usepackage{subcaption}
\usepackage{fancyvrb}
\usepackage{enumerate}
\usepackage{tabularx}
\usepackage{algorithm}
\usepackage{algorithmic}
\usepackage{multirow}
\usepackage{rotating}

\RequirePackage[OT1]{fontenc}
\RequirePackage[numbers]{natbib}
\RequirePackage[colorlinks,citecolor=blue,urlcolor=blue]{hyperref}

\usepackage{./StefanTexCommandsV2}
\usepackage{hyperref,url}
\usepackage{amsmath,amssymb,amsthm}
\usepackage{tikz}
\usepackage{xspace}
\usepackage{stmaryrd,wasysym,clrscode}
\usepackage{etex,etoolbox}
\usepackage{ifthen}
\usetikzlibrary{patterns,positioning}

\usepackage{thmtools}
\usepackage{thm-restate}

\usepackage{hyperref}

\def\[#1\]{\begin{align}#1\end{align}}
\def\(#1\){\begin{align*}#1\end{align*}}

\definecolor{NAColor}{rgb}{.75,0,.75}

\def\argmin{\operatornamewithlimits{arg\,min}}

\newcommand{\bprf}{\begin{proof}}
\newcommand{\eprf}{\end{proof}}
\newcommand{\blem}{\begin{lemma}}
\newcommand{\elem}{\end{lemma}}

\newcommand{\Lstar}{\mathcal{L}}

\DeclareMathOperator{\Regret}{Regret}

\newcommand{\eqdef}{\stackrel{\mathrm{def}}{=}}
\newcommand{\bP}{\mathbb{P}}

\newcommand{\bE}{\mathbb{E}}
\newcommand{\sF}{\mathcal{F}}
\newcommand{\sH}{\mathcal{H}}

\newcommand{\bR}{\mathbb{R}}

\newcommand{\sN}{\mathcal{N}}

\DeclareMathOperator{\supp}{supp}

\newtheorem{theorem}{Theorem}[section]
\newtheorem{lemma}[theorem]{Lemma}

\newtheorem{proposition}[theorem]{Proposition}
\newtheorem{corollary}[theorem]{Corollary}

\newtheorem{assumption}{Assumption}
\theoremstyle{definition}

\graphicspath{{.}}

\newcommand{\sm}[1]{{\scriptstyle #1}}

\newcommand\MainRegret{\Omega}
\newcommand\CostOfSparsity{\Lambda}
\newcommand\Quadratic{Q}
\renewcommand{\lg}{\log_2}

\makeatletter
\renewcommand*{\@fnsymbol}[1]{\ensuremath{\ifcase#1\or  \or \dagger\or \ddagger\or
   \mathsection\or \mathparagraph\or \|\or **\or \dagger\dagger
   \or \ddagger\ddagger \else\@ctrerr\fi}}
\makeatother

\begin{document}


\begin{frontmatter}

\title{The Statistics of \\ Streaming Sparse Regression}
\runtitle{Streaming Sparse Regression}

\begin{aug}
\author{\fnms{Jacob} \snm{Steinhardt}\ead[label=e1]{jsteinhardt@cs.stanford.edu}\thanksref{t1}},
\author{\fnms{Stefan} \snm{Wager}\corref{}\ead[label=e2]{swager@stanford.edu}\thanksref{t1}},
\and
\author{\fnms{Percy} \snm{Liang}\ead[label=e3]{pliang@cs.stanford.edu}}

\address{Department of Computer Science \\ Stanford University \\ Stanford, CA-94305 \\ \printead{e1} \\ \printead{e3}}
\address{Department of Statistics \\ Stanford University \\ Stanford, CA-94305 \\ \printead{e2}}
\affiliation{Stanford University}

\thankstext{t1}{JS and SW contributed equally to this paper. JS is supported by a Hertz Foundation Fellowship and an NSF Fellowship; SW is supported by a BC and EJ Eaves Stanford Graduate Fellowship. We are grateful for helpful conversations with Emmanuel Cand\`es and John Duchi.} 
\end{aug}

\runauthor{Steinhardt, Wager, and Liang}

\begin{abstract}
We present a sparse analogue to stochastic gradient descent that is guaranteed to perform well under
similar conditions to the lasso. In the linear regression setup with irrepresentable noise features,
our algorithm recovers the support set of the optimal parameter vector with high probability,
and achieves a statistically quasi-optimal rate of convergence of $\toop{k\log(d)/T}$,
where $k$ is the sparsity of the solution, $d$ is the number of features, and $T$ is the number
of training examples. Meanwhile, our algorithm does not require any more
computational resources than stochastic gradient descent. In our experiments, we find that our method
substantially out-performs existing streaming algorithms on both real and simulated data.

\end{abstract}

\end{frontmatter}

\section{Introduction}
\label{sec:intro}

In many areas such as
astrophysics \citep{adelman2008sixth,battams2014stream},
environmental sensor networks \citep{osborne2012real},
distributed computer systems diagnostics \citep{xu2009detecting},
and advertisement click prediction \citep{mcmahan2013ad},
a system generates a high-throughput {stream} of data in real-time.
We wish to perform parameter estimation and prediction in this streaming setting,
where we have neither memory to store all the data nor time for complex algorithms.
Furthermore, this data is also typically high-dimensional,
and thus obtaining {sparse} parameter vectors is desirable.
This article is about the design and analysis of statistical procedures that
exploit {sparsity} in the {streaming} setting.

More formally, the streaming setting (for linear regression) is as follows:
At each time step $t$, we (i) observe covariates $x_t \in \RR^d$, (ii) make a
prediction $\hat y_t$ (using some weight vector $w_t \in \RR^d$ which we maintain),
(iii) observe the true response $y_t \in \RR$,
and (iv) update $w_t$ to $w_{t+1}$.
We are interested in two measures of performance after $T$ time steps.
The first is \emph{regret}, which is the excess online prediction error
compared to a fixed weight vector $u \in \RR^d$ (typically chosen to be $w^*$, the population loss minimizer):
\begin{align}
\label{eq:regret}
\Regret(u) \eqdef \sum_{t=1}^T (f_t\p{w_t} - f_t\p{u}),
\end{align}
where $f_t(w) = \frac12(y_t - w^\top x_t)^2$ is the squared loss on the $t$-th data point.
The second is the classic \emph{parameter error}, which is 
\begin{align}
\label{eq:parameterError}
\|\hat w_T - w^*\|_2^2,
\end{align}
where $\hat w_T$ is some weighted average of $w_1, \dots, w_T$. 
Note that, while $\Regret(u)$ appears to measure loss on a training set, 
it is actually more closely related to generalization error, since $w_t$ is 
chosen {before} observing $f_t$, and thus there is no opportunity 
for $w_t$ to be overfit to the function $f_t$.

Although the ambient dimension $d$ is large,
we assume that the population loss minimizer $w^* \in \RR^d$ is a $k$-sparse vector, where $k \ll d$.
In this setting, the standard approach to sparse regression is to
use the lasso \citep{tibshirani1996regression} or basis pursuit \citep{chen1998atomic},
which both penalize the $\LI$ norm of the weight vector to encourage sparsity.
There is a large literature showing that the lasso attains good performance
under various assumptions on the design matrix
\citep[e.g.,][]{meinshausen2009lasso,raskutti2010restricted,raskutti2011minimax,van2008high,van2009conditions,zhao2006model}.
Most relevant to us, \citet{raskutti2011minimax} show that the
parameter error behaves as $\oop{k\log(d) / T}$. 
However, these results require solving a global optimization problem over all the points,
which is computationally infeasible in our streaming setting.

In the streaming setting, an algorithm can only store one training example at a
time in memory, and can only make one pass over the data.
This kind of streaming constraint has been studied in the context of, e.g.,
 optimizing database queries 
\citep{alon1996space,flajolet1985probabilistic,munro1980selection},
 hypothesis testing with finite memory \citep{cover1969hypothesis,hellman1970learning},
and online learning or online convex optimization
\citep[e.g.,][]{bottou1998online,crammer2006online,hazan2007adaptive,littlestone1989wm,shalev2011,shalev2007primal,shalev2011stochastic,shamir2013stochastic}.
This latter case is the most relevant to our setting, and the resulting online algorithms are 
remarkably simple to implement and computationally efficient in practice.
However, their treatment of sparsity is imperfect.
For strongly convex functions \citep{hazan2007logarithmic}, one can ignore sparsity altogether and obtain average regret
$\oo{d \log T / T}$, which is clearly much worse than the optimal rate when $k \ll d$.
One could also ignore strong convexity to obtain average regret $\oo{\sqrt{k \log d / T}}$,
which has the proper logarithmic dependence on $d$, but does not have the optimal dependence on $T$.

\begin{figure}[p]
\includegraphics[width=\textwidth]{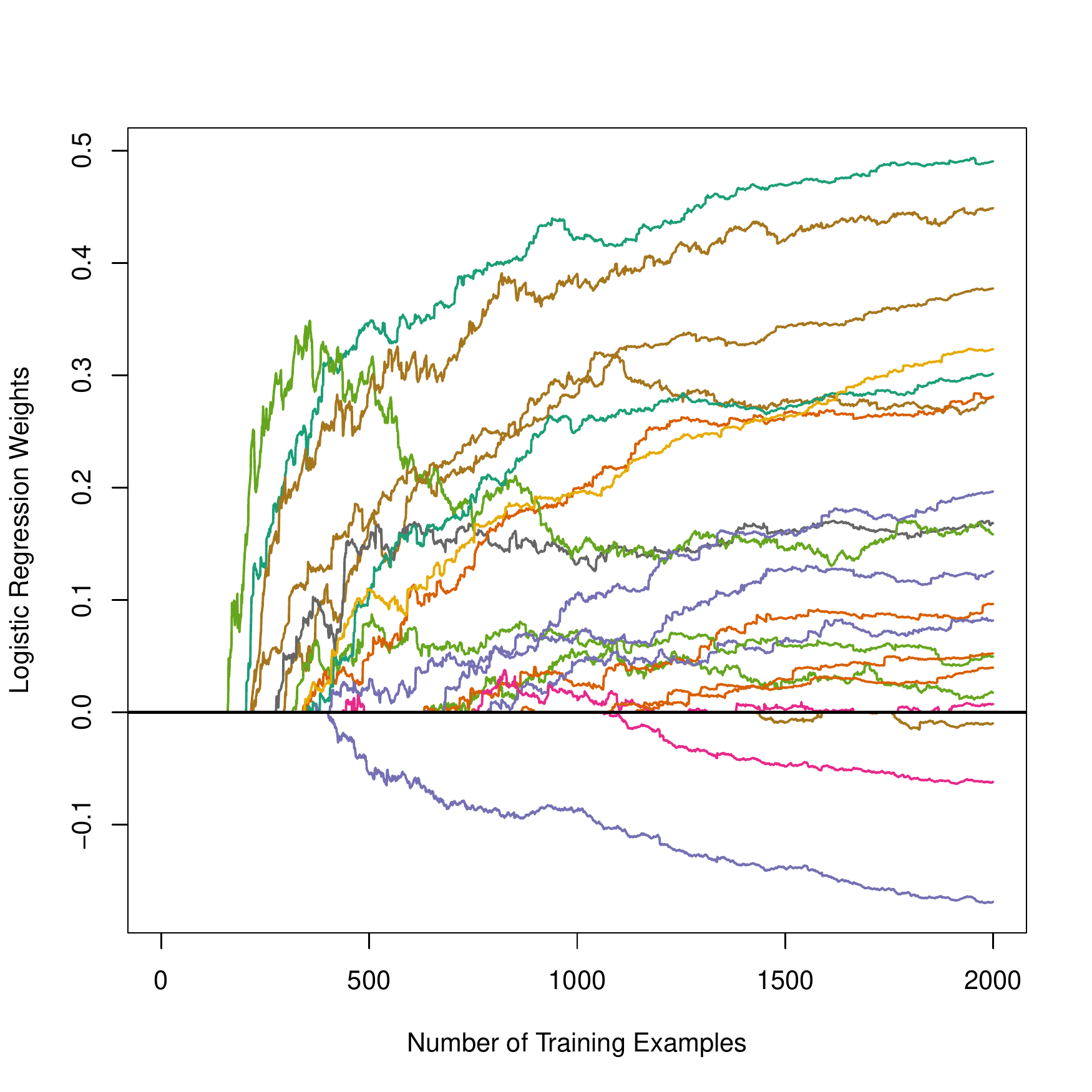}
\caption{Behavior of our Algorithm \ref{alg:stream} as it incorporates the 
first $T = 2,000$ training examples for a logistic regression trained on the 
spambase dataset \citep{bache2013uci}. Due to the streaming nature of the algorithm, 
the parameters are incrementally updated with each new example.
All parameter estimates start at 0; our algorithm then gradually adds variables
to the active set as it sees more training examples and accumulates evidence that
certain variables are informative.
We see that the algorithm found more words with positive weights (i.e., indicative of spam) than
negative weights. In this example, we also used an unpenalized intercept term (not shown)
that was negative. The first positive words selected by the algorithm were \texttt{remove},
\texttt{you}, \texttt{your}, and \texttt{\$}, whereas the first negative words were \texttt{hp}
and \texttt{hpl}; this fits in well with standard analyses \citep{hastie2009elements}. Before
running our algorithm, we centered, scaled, and clipped the features, and randomly re-ordered
the training examples.}
\label{fig:path}
\end{figure}

Our main contribution is an algorithm, streaming sparse regression (SSR),
which takes only $\oo{d}$ time per data point and $\oo{d}$ memory,
but achieves the same convergence rate as the lasso in the batch (offline) setting
under irrepresentability conditions similar to the ones
studied by \citet{zhao2006model}.
The algorithm is very simple, alternating between taking gradients, averaging, and
soft-thresholding.  The bulk of this paper is dedicated
to the analysis of this algorithm, which starts with tools from online convex optimization,
but additionally requires carefully controlling the support of our weight vectors
using new martingale tail bounds.
Recently, \citet{agarwal2012stochastic} proposed a very different
epoch-based $L_p$-norm
algorithm that also attains the desired $\oop{k \log d / T}$ bound on 
the parameter error.
However, unlike our algorithm that is conceptually related to the lasso,
their algorithm does not generate exactly sparse iterates.
Based on our experiments, our algorithm also appears 
to be faster and substantially more accurate in practice.

To provide empirical intuition about our algorithm,
Figure \ref{fig:path} shows its behavior on the spambase
dataset \citep{bache2013uci}, the goal of which is to distinguish spam (1) from non-spam (0)
using 57 features of the e-mail. The plot shows how the parameters change as the algorithm 
sees more data. For the first 159 training examples, all of the weights are zero. Then, as the algorithm 
gets to see more data and amasses more evidence on the association between
various features and the response, it gradually enters new variables into the model.
By the time the algorithm has seen 2000 examples, it has 22 non-zero weights.
A striking difference between Figure \ref{fig:path} and the lasso or least-angle regression
paths of \citet{efron2004least} is that the lasso path moves along straight lines between
knots, whereas our paths look more like Brownian motion once they leave zero. This is 
because Efron et al. vary the $L_1$ regularization for a fixed amount of data, while
in our case the $L_1$ regularization and data size change simultaneously.



\subsection{Adapting Stochastic Gradient Descent for Sparse Regression}

To provide a flavor of our algorithm and the theoretical results involved,
let us begin with classic stochastic gradient descent (SGD),
which is known to work in the non-sparse streaming setting \citep{bottou1998online,robbins1951stochastic,robbins1971convergence,toulis2014statistical}.
Given a sequence of convex loss
functions $f_t(w)$, e.g., $f_t(w) =
\frac12 (y_t - w^\top x_t)^2$ for linear regression with features $x_t$ and response $y_t$,
SGD updates the weight vector as follows:
\begin{equation}
\label{eq:sgd_simple}
w_{t + 1} = w_{t} - \frac{1}{\eta \, t} \, \nabla f_{t}\p{w_{t}} 
\end{equation}
with some step size $\eta > 0$.
As shown by \citet{toulis2014statistical}, if the losses $f_t$ are generated by
a well-conditioned generalized linear model, then the weights $w_t$ will
converge to a limiting Gaussian distribution at a $1/\sqrt{t}$ rate.

\begin{algorithm}[t]
\caption{Streaming sparse regression. $S_{\lambda}$ denotes the soft-thresholding 
operator: $S_{\lambda}(x) = 0$ if $|x| < \lambda$, and $x - \lambda \sign(x)$ otherwise.}
\label{alg:stream}
\begin{algorithmic}
\STATE Input: sequence of loss functions $f_1, \, \ldots, \, f_T$
\STATE Output: parameter estimate $w_T$
\STATE Algorithm parameters: $\eta$, $\lambda$, $\epsilon$
\STATE \vskip -0.1in
\STATE $\theta_1 = 0$
\FOR{$t=1$ {\bfseries to} $T$}
  \STATE $\lambda_t \gets \lambda \sqrt{t+1}$
  \STATE $w_t \gets \frac{1}{\epsilon + \eta \, (t-1)}S_{\lambda_t}\left(\theta_t\right)$ \Comment sparsification step
  \STATE $\theta_{t+1} = \theta_t - [\nabla f_t(w_t) - \eta \, w_t]$ \Comment gradient step
\ENDFOR
\RETURN $w_T$
\end{algorithmic}
\end{algorithm}

While this simple algorithmic form is easy to understand, it is less convenient to extend to exploit sparsity.
Let us then rewrite stochastic gradient
descent using the adaptive mirror descent framework
\citep{beck2009fast,nemirovsky1983problem,orabona2013general}.
With some algebra, one can verify that the update in \eqref{eq:sgd_simple}
is equivalent to the following adaptive mirror descent update:
\begin{align}
\label{eq:sgd_mirror_theta}
&\theta_t = \sum_{s = 1}^{t-1} \nabla f_s\p{w_{s}} \\
\label{eq:sgd_mirror_w}
&w_t = \argmin_w \left\{\frac{\eta}{2} \sum_{s = 1}^{t-1} \Norm{w - w_s}_2^2 + w^{\top}\theta_t\right\}
\end{align}
for $t = 1, \, 2, \, \dots$
At each step, mirror descent solves an optimization problem (usually in closed
form) that (i) encourages weights $w_t$ to be close to previous weights $w_1,
\dots, w_{t-1}$, and (ii) moves towards the average gradient $\theta_t$.

The advantage of using the mirror descent framework
is that it reveals a natural way to induce sparsity:
we can add an $\LI$-penalty to the minimization step
\eqref{eq:sgd_mirror_w}. For some $\lambda > 0$ and $t = 1, \,  2, \, \dots$, we
set
\begin{align}
\label{eq:sparse_mirror_theta}
&\theta_t = \sum_{s = 1}^{t-1} \nabla f_s\p{w_{s}} \\
\label{eq:sparse_mirror_w}
&w_t = \argmin_w \left\{\frac{\eta}{2} \sum_{s = 1}^{t - 1} \Norm{w - w_s}_2^2 + w^{\top}\theta_t + \lambda \sqrt{t+1} \Norm{w}_1 \right\}.
\end{align}
The above update \eqref{eq:sparse_mirror_w} can be efficiently implemented in a streaming setting using
Algorithm \ref{alg:stream}, which is suitable for making online predictions.
We also propose an adaptation (Algorithm \ref{alg:batch}) aimed at classic
parameter estimation;
see Section~\ref{sec:batch} for more details.

These algorithms are closely related to recent proposals in the stochastic
and online convex optimization literature
\citep[e.g.,][]{duchi2010composite,langford2009sparse,shalev2011stochastic,shalev2010trading,xiao2010dual};
in particular, the step \eqref{eq:sparse_mirror_w} can be described as a
proximal version of the regularized dual averaging algorithm of
\citet{xiao2010dual}. These papers, however, all analyze the algorithm making no statistical
assumptions about the data generating process.
Under these adversarial conditions, it is
difficult to provide performance guarantees that take advantage of sparsity.
In fact, the sparsified version of stochastic gradient descent in general
attains weaker worst-case guarantees than even the simple algorithm given in
\eqref{eq:sgd_simple}, at least under existing analyses.

It is well known that the batch lasso works well under some
statistical assumptions
\citep[e.g.,][]{candes2007dantzig,meinshausen2009lasso,raskutti2010restricted,van2008high,van2009conditions,zhao2006model},
but even the lasso can fail spectacularly when these assumptions do not hold,
even for i.i.d. data, e.g., Section 2.1 of \citet{candes2009near}.
It is therefore not surprising that statistical assumptions should also be required 
to guarantee good performance for our streaming sparse regression algorithm.


\begin{algorithm}[t]
\caption{Streaming sparse regression with averaging. $S_{\lambda}$ denotes the soft-thresholding 
operator: $S_{\lambda}(x) = 0$ if $|x| < \lambda$, and $x - \lambda \sign(x)$ otherwise.}
\label{fig:pseudocode}
\label{alg:batch}
\begin{algorithmic}
\STATE Input: sequence of functions $f_1, \, \ldots, \, f_T$
\STATE Output: parameter estimate $\hat w_T$
\STATE Algorithm parameters: $\eta$, $\lambda$, $\epsilon$
\STATE \vskip -0.1in
\STATE $\hat{w}_0 = 0$, $\theta_1 = 0$
\FOR{$t=1$ {\bfseries to} $T$}
  \STATE $\lambda_t \gets t^{\frac{3}{2}} \lambda$
  \STATE $w_t \gets \frac{1}{\epsilon + \eta \, t(t-1) / 2}S_{\lambda_t}\left(\theta_t\right)$ \Comment sparsification step
  \STATE $\theta_{t+1} \gets \theta_t - t\left[\nabla f_t(w_{t}) - \eta \, w_t\right]$ \Comment gradient step
  \STATE $\hat{w}_t \gets \left(1-\frac{2}{t+1}\right)\hat{w}_{t-1} + \frac{2}{t+1}w_t$ \Comment averaging step
\ENDFOR
\RETURN $\hat{w}_T$
\end{algorithmic}
\end{algorithm}

The following theorem gives a flavor (though not the strongest) for the kind of results proved in this
paper, using simplified assumptions and restricting attention to linear regression. 
As we will show later, the orthogonality constraint on the non-signal features is not
in fact needed and an irrepresentability-like condition on the design is enough.

\begin{theorem}[parameter error with uncorrelated noise]
\label{theo:first}
Suppose that we are given an \emph{i.i.d.} sequence of data points $(x_1, \, y_1), \, (x_2, \,  y_2), \dots \in \RR^d \times \RR$ 
satisfying $y_t = (w^*)^\top x_t + \varepsilon_t$, where $S \eqdef \supp(w^*)$
has size $k$ and $\varepsilon_t$ is centered noise.
Let $x_{t}\sm{[S]}$ denote the coordinates of $x_t$ indexed by $S$ and $x_{t}\sm{[\neg S]}$ the 
coordinates in the complement of $S$. Also suppose that 
\begin{align*}
&\EE{\varepsilon_tx_t} = 0, && \EE{x_{t}\sm{[\neg S]} y_{t}} = 0,\\
&\EE{x_{t}\sm{[\neg S]}x_{t}\sm{[S]}^{\top}} = 0, && \lambda_{\min}\left(\EE{x_{t}\sm{[S]}x_{t}\sm{[S]}^{\top}}\right) > 0
\end{align*}
for all $t \in \NN$, where $\lambda_{\min}(M)$ denotes the smallest eigenvalue of $M$.
Then for sufficiently large $\lambda$ and sufficiently small $\eta$, 
if we run Algorithm~\ref{alg:batch} on $\{(x_t, \, y_t)\}_{t=1}^T$, 
with the squared loss $f_t(w) = \frac12 (y_t - w^\top x_t)^2$, 
we will obtain a parameter vector $\hat{w}_T$ with $\supp(\hat w_T) \subseteq S$ satisfying
\begin{equation}
\label{eq:bound_first}
\|\hat{w}_T-w^*\|_2^2 = \oop{\frac{k\log(d\log(T))}{T}},
\end{equation}
where $\mathcal{O}_P$ is a with-high-probability version of $\mathcal{O}$ notation.\footnote{More
specifically, in this paper, we use the notation
$x(T) = \mathcal{O}_P(y(T))$ if $x(T) \leq cy(T)\log(1/\delta)$ with probability $1-\delta$, 
for some constant $c$ that is independent of $T$ or $\delta$.}
\end{theorem}

The bound \eqref{eq:bound_first} matches the minimax optimal rate
for sparse regression when $d \gg k$ \citep{raskutti2011minimax}, namely
\begin{equation}
\label{eq:minimax}
\|\hat{w}_T-w^*\|_2^2 =\oop{\frac{k\log(d)}{T}},
\end{equation}
to within a factor of $1 + \frac{\log \log(T)}{\log (d)}$,
which is effectively bounded by a constant since $\log \log(T) / \log(d) \leq 5$
in any reasonable regime.\footnote{The extra $\log\log T$ term can be understood in terms of the
law of the iterated logarithm. Our results requires us to bound the behavior of
the algorithm for all $t = 1, \, ..., \,  T$; thus,  we need to analyze multiple $t$-scales
simultaneously, and an extra $\log\log T$ term appears. This is exactly the same phenomenon
that arises when we study the scaling of the limsup of a random walk: although the pointwise
distribution of the random walk scales as $\sqrt{T}$, the limsup scales as $\sqrt{T \log\log T}$.}

\subsection{Related work}

There are many existing online algorithms for solving 
optimization problems like the lasso. For each of these, we will state their rate 
of convergence in terms of the rate at which the squared parameter error $\|\hat{w}_T-w^*\|^2_2$ 
decreases as we progress along an infinite stream of i.i.d. data.
As discussed above, the simplest online algorithm is the classical stochastic gradient descent 
algorithm, which achieves error $\mathcal{O}({{d}/{T}})$ under statistical assumptions. A later family 
of algorithms, comprising the exponentiated gradient algorithm \citep{kivinen1997} and the family of $p$-norm 
algorithms \citep{gentile2003robustness}, achieves error $\mathcal{O}(\sqrt{{k\log(d)}/{T}})$; while 
$d$ has been replaced by $k\log(d)$, the algorithm no longer achieves the optimal rate in $T$.

There is thus a tradeoff in existing work
between better dependence on the dimension and worse asymptotic convergence.
In contrast, our approach simultaneously achieves good performance in terms 
of both $d$ and $T$. 
Given statistical assumptions, our algorithm
satisfies tighter excess loss bounds 
than existing sparse SGD-like algorithms
\citep[e.g.,][]{duchi2010composite,langford2009sparse,mcmahan2011follow,shalev2011stochastic,shalev2010trading,xiao2010dual}.
\citet{agarwal2012stochastic} obtain similar theoretical bounds to us
using a very different algorithm, namely an epoch-based $L_p$-norm
regularized mirror descent algorithm. In our experiments, it appears that
our more lasso-like streaming algorithm achieves better performance, both
statistically and computationally.

In other work, \citet{gerchinovitz2013sparsity} derived strong adversarial ``sparsity regret
bounds'' for an exponentially weighted Bayes-like algorithm with a heavy-tailed
prior. However, as its implementation requires the use of Monte-Carlo methods,
this algorithm may not be computationally competitive with efficient
$\LI$-based methods. There has also been much work beyond that already discussed 
on solving the lasso in the online or streaming setting, such as 
\citet{garrigues2009homotopy} and \citet{yang2010online}, but none of these 
achieve the optimal rate.

Finally, we emphasize that there are paradigms other than streaming for doing regression on large datasets.
In lasso-type problems, the use of pre-screening rules to remove variables from consideration can
dramatically decrease practical memory and runtime requirements. Some examples 
include strong rules \citep{tibshirani2012strong} and SAFE rules \citep{el2010safe}.
Meanwhile, \citet{fithian2014local} showed that, in locally imbalanced logistic regression problems,
it is often possible to substantially down-sample the training set without losing much statistical information;
see also \citep{alaoui2014fast,raskutti2014statistical} for related ideas.
Comparing the merits of streaming algorithms to those of screening or subsampling
methods presents an interesting topic for further investigation.

\subsection{Outline}

We start in Section~\ref{sec:main} by 
precisely defining our theoretical setting and providing our main theorems with some intuitions.
We then demonstrate the empirical performance of our algorithm on simulated data
(Section~\ref{sec:simulations}) and a genomics dataset
(Section~\ref{sec:application}).
In Section~\ref{sec:sparse},
we use the adaptive mirror descent framework from online convex
optimization to lay the foundation of our analysis.
We then leverage statistical assumptions to provide tight control over the
terms laid out by the framework (Sections~\ref{sec:statistical} and \ref{sec:expected_cvx}),
resulting in bounds on the prediction error of Algorithm~\ref{alg:stream}.
In Section~\ref{sec:batch}, we adapt our algorithm via weighted averaging
to obtain rate-optimal parameter estimates (Algorithm~\ref{alg:batch}).
Finally, in Section~\ref{sec:irrep},
we weaken our earlier assumptions to an irrepresentability condition similar
to the one given in \citet{zhao2006model}. 
Longer proofs are deferred to the appendix.


\section{Statistical Properties of Streaming Sparse Regression}
\label{sec:main}

\subsection{Theoretical Setup}
\label{sec:setup}

We assume that we are given a sequence of loss functions $f_1, \, f_2,$
$\ldots, \, f_T$ drawn from some joint distribution.
Our algorithm produces a sequence $w_1, \, w_2,$ $\ldots, \, w_T$, where each $w_t$ 
depends only on $f_1, \, f_2$, $\ldots, \, f_{t-1}$.

Our main results depend on the following four assumptions.

\begin{enumerate}
\item {\bf Statistical Sparsity:} There is a fixed expected loss function $\Lstar$ such that
$$ \EE{f_t \mid f_1, \, \dots, \, f_{t - 1}} = \Lstar \; \eqfor \; t = 1, \, 2, \, \dots$$
Moreover, the minimizer $w^*$ of the loss $\Lstar$ satisfies $\|w^*\|_1 \leq R$ and $\supp(w^*) = S$, where $|S| \leq k$.
Define the set of candidate weight vectors:
$$\sH \eqdef \{w : \|w\|_1 \leq R, \supp(w) \subseteq S\}.$$
We note that $\sH$ is not directly available to the statistician, because
she does not know $S$.

\item {\bf Strong Convexity in Expectation:} There is a constant
$\alpha > 0$ such that $\Lstar(w) - \frac{\alpha}{2}\|w\sm{[S]}\|_2^2$ is
convex. Recall that, for an arbitrary vector $w$, $w\sm{[S]}$ denotes the
coordinates indexed by $S$ and $w\sm{[\neg S]}$ denotes the remaining
coordinates.

\item {\bf Bounded Gradients:} The gradients $\nabla
f_t$ satisfy $\|\nabla f_t(w)\|_{\infty} \leq B$ for all $w \in \sH$.

\item {\bf Orthogonal Noise Features:} For our simplest results, we assume that
the noise gradients are mean-zero for all $w \in \sH$: more precisely, for all
$i \not\in S$ and all $w \in \sH$, we have $\nabla \Lstar(w)_i = 0$. In
Section \ref{sec:irrep_front} below, we discuss how we can relax this condition into an
irrepresentability condition.
\end{enumerate}

To gain a better understanding of the meaning of these assumptions, we give some simple
conditions under which they hold for linear regression. Recall that in linear regression, we are given a
sequence of examples $(x_t, \, y_t) \in \RR^d \times \RR$, and have a loss function 
$f_t(w) = \frac{1}{2}(y_t-w^{\top}x_t)^2$. Here, the assumption (1) holds if the $(x_t, \, y_t)$ are 
i.i.d. and the minimizer of
$$\Lstar(w) = \bE\left[\frac{1}{2}(y-w^{\top}x)^2\right] $$
is $k$-sparse. Meanwhile, we can check that $\Lstar$ is a quadratic function with leading term
$\frac{1}{2}w^{\top}\EE{xx^{\top}}w, $
and so (2) holds as long as $\Cov{x\sm{[S]}} \succeq \alpha I$. Next, $\nabla f_t(w) = (y-w^{\top}x)x$, 
so $\|\nabla f_t(w)\|_{\infty} \leq |y_t| \|x_t\|_{\infty} + \|w\|_1\|x_t\|_{\infty}^2$. Hence, 
if we assume that $\|x_t\|_{\infty} \leq B_x$ and $|y_t| \leq B_y$, assumption (3) 
holds with $B = B_xB_y + RB_x^2$.

The most stringent condition is assumption (4), which requires that $\bE[(y-w^{\top}x)x_i] = 0$ for 
all $i \not\in S$ and $w \in \sH$. A sufficient condition is that $\bE[yx\sm{[\neg S]}] = 0$ 
and $\bE[x\sm{[S]}x\sm{[\neg S]}^{\top}] = 0$, i.e., the noise coordinates are mean-zero and 
uncorrelated with both $x\sm{[S]}$ and $y$. Assumption 4 can, however, in general be relaxed.
For example, in the case of linear regression, we can replace it with an irrepresentability condition
(Section \ref{sec:irrep_front}).

\subsection{Main Results}

We presents two results that control the two quantities of interest:
(i) the regret (\ref{eq:regret}) with respect to the population loss minimizer
$w^*$, which evaluates prediction; and (ii) the parameter error $\|\hat w_T - w^*\|_2^2$.

The first result controls $\Regret(w^*)$ for Algorithm \ref{alg:stream}; 
the bulk of the proof involves 
showing that our $\LI$ sparsification step succeeds at keeping the noise
coordinates at zero without incurring too much extra loss.

\begin{theorem}[online prediction error with uncorrelated noise]
\label{thm:main-risk}
Suppose that the sequence $f_1, \, \ldots, \, f_T$ satisfies assumptions (1-4) from Section \ref{sec:setup} and that we use Algorithm \ref{alg:stream} with
$$\lambda = \frac{3B}{2}\sqrt{\log\p{\frac{6d\log_2(2T)}{\delta}}}, $$
$\eta = \alpha/2$,
and $\epsilon = 0$. Then, for any $\delta > 0$, with probability 
$1-\delta$, we have
\begin{align}
\label{eq:main-risk}
\Regret(w^*)
 = \oo{\frac{kB^2}{\alpha}\log\p{\frac{d\log(T)}{\delta}}\log(T)}.
\end{align}
\end{theorem}

The second result controls $\Norm{\hat{w}_T - w^*}_2^2$, where $\hat{w}_T$ is the weighted average 
given in Algorithm \ref{alg:batch}.
To transform Theorem \ref{thm:main-risk} into a parameter error bound, we use a standard
technique: online-to-batch conversion \citep{cesa2004generalization}.
As we will discuss in Section \ref{sec:batch}, a naive application of
online-to-batch conversion to Algorithm \ref{alg:stream} yields a result that
is loose by a factor of $\log (T)$. Thus, in order to bound batch error we need
to modify the algorithm, resulting in Algorithm \ref{alg:batch} and the
following bound:

\begin{theorem}[parameter error with uncorrelated noise]
\label{thm:main-param}
Suppose that we make the same assumptions and parameter 
choices as in Theorem~\ref{thm:main-risk}, except that we now set 
$$\lambda = \frac{3B}{2}\sqrt{\log\p{\frac{6d\log_2(2T^3)}{\delta}}}. $$
Let $\hat{w}_T$ be the output of Algorithm \ref{alg:batch}.
Then, with probability $1-\delta$, 
we have $\supp\p{\hat{w}_T} \subseteq S$ and 
\begin{align}
\label{eq:main-param}
\|\hat{w}_T-w^*\|_2^2
 &= \oo{\frac{kB^2}{\alpha^2T}\log\p{\frac{d\log(T)}{\delta}}}
\end{align}
for any $\delta > 0$.
\end{theorem}

\subsection{Irrepresentability and Support Recovery}
\label{sec:irrep_front}

In practice, Assumption 4 from Section \ref{sec:setup} is unreasonably strong:
in the context of high-dimensional regression, we cannot in general hope for the noise features to
be exactly orthogonal to the signal ones. Here, we discuss how this condition can be relaxed in the
context of online linear regression.

In the batch setting, there is a large literature on establishing
conditions on the design matrix $X \in \RR^{n \times d}$ under which the lasso performs well
\citep[e.g.,][]{meinshausen2009lasso,raskutti2010restricted,van2008high,van2009conditions,zhao2006model}.
The two main types of assumptions typically made on the design $X$ are as follows:
\begin{itemize}
\item The {\bf restricted eigenvalue condition} \citep{bickel2009simultaneous,raskutti2010restricted}
is sufficient for obtaining low $\LII$ prediction error under sparsity assumptions on $w^*$.
A similar condition is also necessary in the minimax setting \citep{raskutti2011minimax}.
\item The stronger {\bf irrepresentability condition} \citep{meinshausen2006high,zhao2006model} is
sufficient and essentially necessary for recovering the support of $w^*$.
\end{itemize}
We will show that our Algorithm \ref{alg:batch}
still converges at the rate \eqref{eq:main-param} under a slight strengthening of the
standard irrepresentability condition, given below: 

\setcounter{assumption}{4}
\begin{assumption}[irrepresentable noise features]
\label{assu:irrep_noise}
The noise features are irrepresentable 
using the signal features in the sense that, 
for any $\tau \in \bR^d$ with $\supp(\tau) \subseteq S$ and any $j \notin S$,
\begin{equation}
\label{eq:irrep_noise}
\Abs{\Cov{x_t^j, \, \tau \cdot x_t}} \leq \rho \frac{\alpha}{\sqrt{k}} \Norm{\tau}_2
\end{equation}
for some constant $0 \leq \rho < 1/\sqrt{24}$.
Recall that $\alpha$ is the strong convexity parameter of the expected loss,
and $|S| = k$.
\end{assumption}

The fact that our algorithm requires an irrepresentability condition instead of the
weaker restricted eigenvalue condition stems from the fact that our algorithm effectively
achieves low prediction error via support recovery; see, e.g., 
Lemma~\ref{lem:irrep_lambda}. Thus, we need conditions on the design $X$ that are
strong enough to guarantee support recovery. For an overview of how different assumptions
on the design relate to each other, see \citet{van2009conditions}.

Given Assumption \ref{assu:irrep_noise}, we have the following bound on
the performance of Algorithm \ref{alg:batch}. We show how Theorem \ref{thm:main-param}
can be adapted to yield this result in Section \ref{sec:irrep}.

\begin{theorem}[parameter error with irrepresentability]
\label{theo:irrep_main}
Under the conditions of Theorem \ref{thm:main-param}, suppose that we replace
Assumption 4 from Section \ref{sec:setup} with the irrepresentability
Assumption \ref{assu:irrep_noise} above. Then, for any $\delta > 0$, for an
appropriate setting of $\lambda$ we have 
\begin{align}
\label{eq:irrep_main}
\|\hat{w}_T-w^*\|_2^2
 &= \oo{\frac{1}{1 - 24\rho^2} \, \frac{kB^2}{\alpha^2T}\log\p{\frac{d\log(T)}{\delta}}}
\end{align}
with probability $1 - \delta$.
\end{theorem}

A form of the standard irrepresentability condition for the batch lasso that only depends
on the design $X$ is given by \citep{van2009conditions}:
\begin{equation}
\label{eq:irrep_batch}
\max_{\tau \in \{-1, \, +1\}^k} \Norm{\Sigma_{\neg S, \, S} \Sigma_{S, \, S}^{-1} \, \tau}_\infty < 1,
\end{equation}
where $\Sigma = \Var{X}$, $\Sigma_{S, \, S}$ is the variance of the signal coordinates of $X$, 
and $\Sigma_{\neg S, \, S}$ is the covariance between the non-signal and signal coordinates.
The conditions \eqref{eq:irrep_noise} and \eqref{eq:irrep_batch} are within a constant factor of each other if none
of the entries of $\Sigma_{\neg S, \, S}$ are much bigger than the others; for example, in the
equicorrelated case, they both require the cross-term correlations to be on the order of $1/\sqrt{k}$.
On the other hand, \eqref{eq:irrep_batch} allows $\Sigma_{\neg S, \, S}$ to have a small number of larger
entries in each row, whereas \eqref{eq:irrep_noise} does not. It seems plausible to us that an
analogue to Theorem \ref{theo:irrep_main} should still hold under a weaker condition that
more closely resembles \eqref{eq:irrep_batch}.

\subsection{Proof Outline and Intuition}
\label{sec:outline}

Our analysis starts with results from online convex optimization that study
a broad class of \emph{adaptive mirror descent updates}, which have the following general form:
\begin{equation}
\label{eq:mirror}
w_t = \argmin_{w} \left\{\psi_t(w) + w^{\top} \theta_t \right\}, \; \where \; \theta_t = \sum_{s=1}^{t-1} \nabla f_s(w_s)
\end{equation}
and $\psi_t$ is a convex regularizer.
Note that our method from Algorithm \ref{alg:stream} is an instance of adaptive mirror descent with
the regularizer
\begin{equation}
\label{eq:psi1}
\psi_t\p{w} =  \frac{\epsilon}{2}\Norm{w}_2^2 + \frac{\eta}{2} \sum_{s = 1}^{t - 1} \Norm{w - w_s}_2^2 +  \lambda \sqrt{t+1} \Norm{w}_1.
\end{equation}
The following result by \citet{orabona2013general} applies to all procedures of the form \eqref{eq:mirror}:

\begin{proposition}[adaptive mirror descent \cite{orabona2013general}]
\label{prop:adaptive}
Let $f_t(\cdot)$ be a sequence of loss functions, let $\psi_t(\cdot)$ be a sequence 
of convex regularizers, and let $w_t$ be
defined as in \eqref{eq:mirror}. Then, for any $u \in \RR^d$,
\begin{align}
\label{eq:adaptive}
\sum_{t=1}^T &\p{w_t-u}^{\top}\nabla f_t(w_t)
 \\ \notag
&\leq \psi_{T}(u) + \sum_{t=1}^T D_{\psi^*_t}\p{\theta_{t + 1} || \theta_{t}} + \sum_{t=1}^T [\psi_{t-1}(w_t) - \psi_t(w_t)].
\end{align}
Here, we let $\psi_0(\cdot) \equiv 0$ by convention and use 
$D_{\psi^*_t}$ to denote the \emph{Bregman divergence}:
\begin{equation}
\label{eq:bregman}
 D_{\psi^*_t}\p{\theta_{t + 1} || \theta_{t}} = \psi^*_t\p{\theta_{t + 1}} - \psi^*_{t}\p{\theta_t} - \left\langle \nabla \psi^*_t \p{\theta_{t}}, \, \theta_{t + 1} - \theta_t \right\rangle.
\end{equation}
\end{proposition}

The bound \eqref{eq:adaptive} is commonly used when the losses $f_t$ are convex,
in which case we have:
\begin{equation}
\label{eq:cvx_motiv}
f_t(w_t) - f_t(u) \leq (w_t-u)^{\top}\nabla f_t(w_t),
\end{equation}
which immediately results in an upper bound on $\Regret(u)$.
We emphasize, however, that \eqref{eq:adaptive} still holds even when $f_t$ is not convex;
we will use this fact to our advantage in Section \ref{sec:expected_cvx}.

Proposition \ref{prop:adaptive} turns out to be very powerful. As shown by \citet{orabona2013general}, 
many classical online learning bounds that were originally proved using ad-hoc methods 
follow directly as corollaries of \eqref{eq:adaptive}. This framework has also led to 
improvements to existing algorithms \citep{steinhardt2014adaptivity}.
Applied in our context, and setting $u = w^*$, we obtain the following bound (see the 
appendix for details):

\begin{corollary}[decomposition]
\label{cor:adaptive}
If we run Algorithm~\ref{alg:stream} on loss functions $f_1, \, \ldots, \, f_T$,
then for any $u \in \sH$ (in particular, $u = w^*$):
\begin{align}
\label{eq:meta}
\sum_{t=1}^T &(w_t-u)^{\top}\nabla f_t(w_t) \le \MainRegret_0 + \CostOfSparsity + \Quadratic, \\
\label{eq:MainRegret0}
\MainRegret_0 &\eqdef \frac{\epsilon}{2}\|u\|_2^2
 + \frac{1}{2}\sum_{t=1}^T \frac{\|\nabla f_t(w_t)\|_2^2}{\epsilon + \eta t}, \\
\label{eq:CostOfSparsity}
\CostOfSparsity &\eqdef \sum_{t=1}^T \p{\lambda_{t-1}-\lambda_t}\p{\|w_t\|_1 - \|u\|_1}, \\
\label{eq:Quadratic}
\Quadratic &\eqdef \frac{\eta}{2}\sum_{t=1}^T \|w_t-u\|_2^2.
\end{align}
\end{corollary}

In words, Corollary~\ref{cor:adaptive} says that the linearized regret is upper
bounded by the sum of three terms:
(i) the main term $\MainRegret_0$ that roughly corresponds to performing
stochastic gradient descent under sparsity from the $L_1$ penalty,
(ii) the cost of ensuring that sparsity $\CostOfSparsity$,
and (iii) a final quadratic term, that will be canceled out by strong convexity
of the loss. 

To achieve our goal from Theorem~\ref{thm:main-risk} of showing that
\begin{align}
\label{eq:main-risk-repeat}
\Regret(w^*) = \oop{k\log(d\log(T))\log(T)},
\end{align}
it remains to control each of the three terms in (\ref{eq:meta}).
The rest of this section provides a high-level overview of 
our argument, indicating where the 
details of the proof appear in the remainder of the paper.

\subsection*{Enforcing Sparsity}

The first problem with \eqref{eq:meta} is that the norms $\|\nabla f_t(w_t)\|_2^2$ in \eqref{eq:MainRegret0} in general
scale with $d$, which is inconsistent with the desired bound \eqref{eq:main-risk-repeat}, which only scales with $\log d$.
In Section \ref{sec:sparse}, we establish a strengthened version of Proposition \ref{prop:adaptive}
that lets us take advantage of effective sparsity of the weight vectors $w_t$ by restricting the
Bregman divergences from \eqref{eq:bregman} to a set of active features. Thanks to our noise assumptions (4) or (5) paired
with an $\LI$ penalty that scales as $\sqrt{t}$, we can show that our active set will have size at most
$k$ with high probability.  This implies that we can replace the term $\MainRegret_0$ in Corollary
\ref{cor:adaptive} with a new term $\MainRegret$ that scales as $\oop{k \log T}$.

\subsection*{Bounding the Cost of Sparsity}

Second, we need to bound the cost of sparsity $\CostOfSparsity$.
A standard analysis following the lines of, e.g., \citet{duchi2010composite} would use 
the inequality 
$\p{\|w_t\|_1 - \|w^*\|_1} \geq -R$, thus resulting in a bound on the cost of $\LI$ penalization
$\CostOfSparsity$ that scales as $R\sqrt{T}$, which again is too large for our purposes.

In a statistical setup, however, we can do better. We know that
$|\lambda_{t-1}-\lambda_t| \approx {\lambda} / (2{\sqrt{t}})$. Meanwhile, given adequate
assumptions, we might also hope for $|\|w_t\|_1-\|w^*\|_1|$ to decay at a rate of
${k}/{\sqrt{t}}$ as well. Combining these two bounds would bound the cost of sparsity
on the order of $\lambda k \log T$. 

The difficulty, of course, is that obtaining bounds of $|\|w_t\|_1-\|w^*\|_1|$ requires controlling
the cost of sparsity, resulting in a seemingly problematic recursion. In Section \ref{sec:statistical},
we develop machinery that lets us simultaneously bound $|\|w_t\|_1-\|w^*\|_1|$ and the cost
of sparsity $\CostOfSparsity$, thus letting us break out of the circular argument. 
The final bound on $\Lambda$ involves a multiplicative constant of $\lambda^2$, where $\lambda$ must be at least 
$\sqrt{\log(d\log(T))}$, which is where the $\log(d\log(T))$ term in our bound comes from.

Finally, we emphasize that our bound on the cost of sparsity crucially depends on $\lambda_t$
growing with $t$ in a way that keeps $\lambda_t - \lambda_{t - 1}$ on a scale
of at most $1/\sqrt{t}$.
Existing methods \citep{duchi2010composite,shalev2011stochastic,xiao2010dual} often just use
a fixed $\LI$ penalty $\lambda_t = \lambda$ for all $t$.
To ensure sparsity, this requires $\lambda$ to be on the order of $\sqrt{T}$, which would 
in turn impose a cost of sparsity of $\sqrt{T}$, rather than the $\log(T)$ cost that we seek.

\subsection*{Working with Strong Convexity in Expectation}
Finally, we need to account for the quadratic term $\Quadratic$ given in \eqref{eq:Quadratic}.
If we knew that $f_t$ were $\alpha$-strongly convex for all $t$, then by definition,
\begin{equation}
\label{eq:ineq-convex}
\sum_{t = 1}^T \p{f_t(w_t) - f_t(w^*)} + \frac{\alpha}{2}\sum_{t=1}^T \|w_t-w^*\|_2^2 \leq \sum_{t = 1}^T (w_t-w^*)^{\top}\nabla f_t(w_t).
\end{equation}
Thus, provided that $\eta \leq \alpha$, we could remove the term \eqref{eq:Quadratic}
when using \eqref{eq:meta} to establish an excess risk bound.

In our application, only the expected loss $\law(w)$ as defined
in Assumption (1) is $\alpha$-strongly convex; the
loss functions $f_t$ themselves are in general not strongly convex. In Section
\ref{sec:expected_cvx}, however, we show that we can still obtain a high-probability
analogue to \eqref{eq:ineq-convex} when $f_t$ is strongly convex in expectation, provided
that $\eta \leq \alpha / 2$.

Putting all these inequalities together, we can successfully 
bound all terms in \eqref{eq:meta} by $\oop{k\log(d\log(T))\log(T)}$.
The last part of our paper then extends these results to provide
bounds for the parameter error of Algorithm \ref{alg:batch} (Section \ref{sec:batch}),
and adapts them to the case of irrepresentable instead of orthogonal features
(Section \ref{sec:irrep}).

\section{Experiments}
\label{sec:experiments}

To test our method, we ran it on several simulated datasets and
a genome-wide association study, while comparing it to several existing methods.
The streaming algorithms we considered were:
\begin{enumerate}
\item Our method, streaming sparse regression (SSR), given in Algorithm \ref{alg:stream},
\item $p$-norm regularized dual averaging ($p$-norm + $\LI$) \citep{shalev2011stochastic}, which exploits sparsity but not strong convexity, and
\item The epoch-based algorithm of \citet*{agarwal2012stochastic} (ANW), which has theoretically optimal asymptotic rates.
\end{enumerate}
We also tried running un-penalized stochastic gradient descent, which exploits strong
convexity but not sparsity; however, this performed badly enough that we did not add it
to our plots.

We also compare all the streaming methods to the batch lasso, which we treat as an oracle.
The goal of the this comparison is to show that, in large-scale problems, streaming algorithms
can be competitive with the lasso. The way we implemented the lasso oracle is by running
\texttt{glmnet} for \textsc{matlab} \citep{friedman2010regularization,qian2013glmnet} with
the largest number of training examples the software could handle before crashing. In both
the simulation and real data experiments, \texttt{glmnet} could not handle all the available data,
so we downsampled the training data to make the problem size manageable; we had to downsample
to $2,500$ out of $10,000$ data points in the simulations and $500$ out of $3,500$ in the genetics example.

\subsection{Simulated Data}
\label{sec:simulations}

We created three different synthetic datasets; for the first two, we ran linear regression 
with a Huberized loss\footnote{Since \texttt{glmnet} does not have an 
option to use the Huberized loss, we used the squared loss instead.}
\[ f_t(w) = h(y_t-w^{\top}x_t), \; h(y) = \left\{ \begin{array}{ccl} {y^2}\big/{2} & : & |y| < C \\ C \cdot (|y|-C/2) & : & |y| \geq C \end{array} \right.. \]
For the third dataset, we used the logistic loss for all methods. 
Our datasets were as follows:
\begin{itemize}
\item \textbf{linear regression, i.i.d. features}: we sampled $x_t \sim \sN(0, I)$ and $y_t = (w^*)^{\top}x_t + v_t$, where 
      $v_t \sim \sN(0, \sigma^2)$, and $w^*$ was a $k$-sparse vector drawn from a Gaussian distribution.
\item \textbf{linear regression, correlated features}: the output relation is the same as before, but now the coordinates of 
      $x_t$ have correlations that decay geometrically with distance 
      (specifically, $\Sigma_{i,j} = 0.8^{|i-j|}$). In addition, the non-zero entries of $w^*$ were 
      fixed to appear consecutively.
\item \textbf{logistic regression}: $x_t$ is a random sign vector and $y_t \in \{0,1\}$, with $p(y_t = 1 \mid x_t) = \tfrac{1}{1+\exp(-(w^*)^{\top}x_t)}$.
\end{itemize}
In each case, we generated data with $d = 100,000$. The first $k = 100$ entries of $w^*$ were drawn from 
independent Gaussian random variables with standard deviation 0.2; the remaining 99,900 entries were 0. 

\begin{table}
\caption{Average runtime (seconds)}
\label{fig:timing}
\begin{tabular}{ccccc}
  &    i.i.d &  correlated &  logit &  gene \\
\hline 
SSR  &    11.3 &  12.1  &  12.2 &  29.2 \\
$p$-norm & 131.5 & 114.3  &  77.7 & 122.0 \\
ANW  &   340.9 & 344.4 &  351.9 & 551.9 
\end{tabular}
\end{table}

\newcommand{\FIGH}{0.3}

\begin{figure}
\begin{center}
\centerline{\begin{tabular}{rcc}
& Prediction Error & Parameter Error \\
{\begin{sideways}\parbox{\FIGH\textheight}{\centering Linear Regression, i.i.d. Features}\end{sideways}}&
\includegraphics[height=\FIGH\textheight]{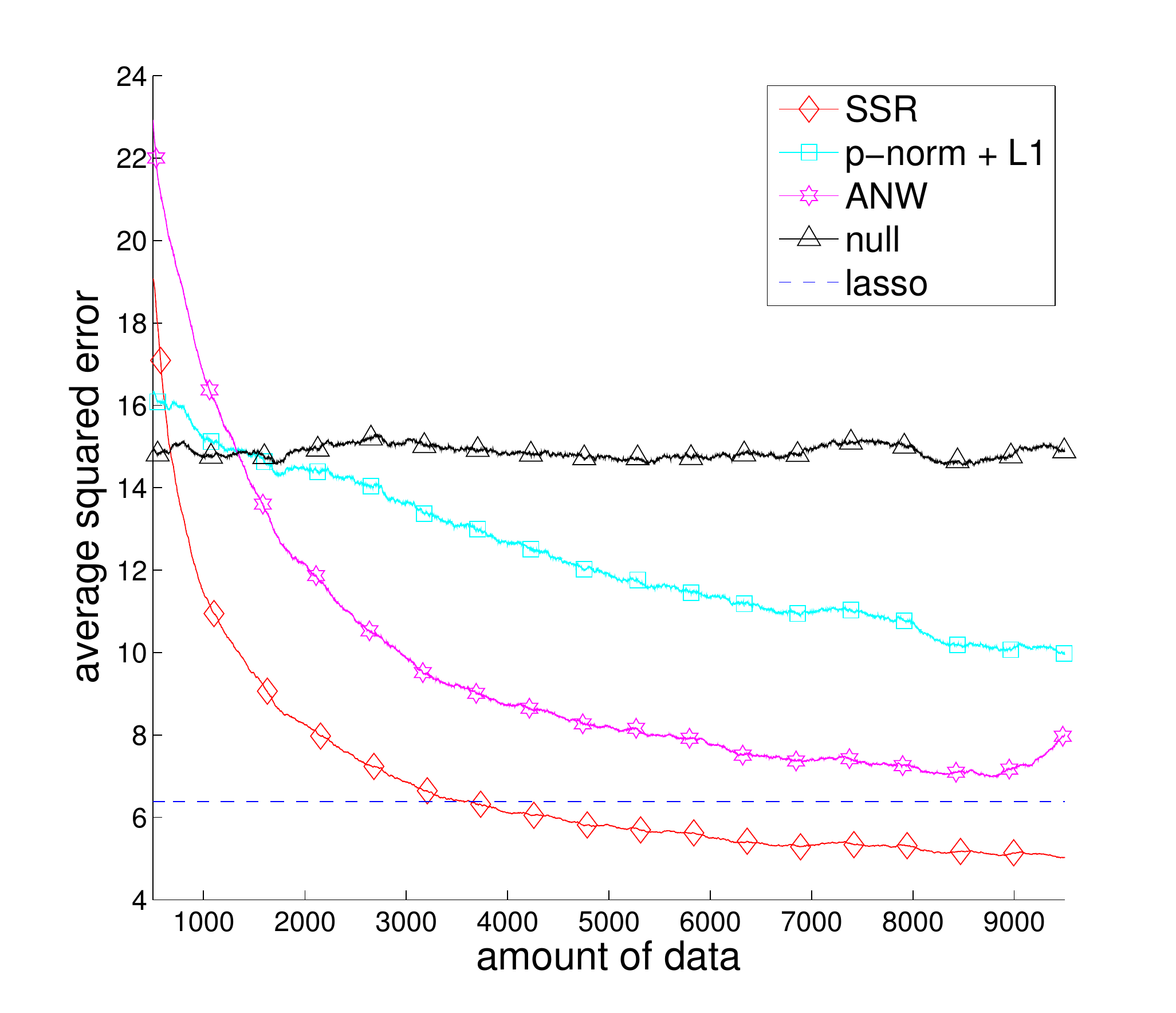} &
\includegraphics[height=\FIGH\textheight]{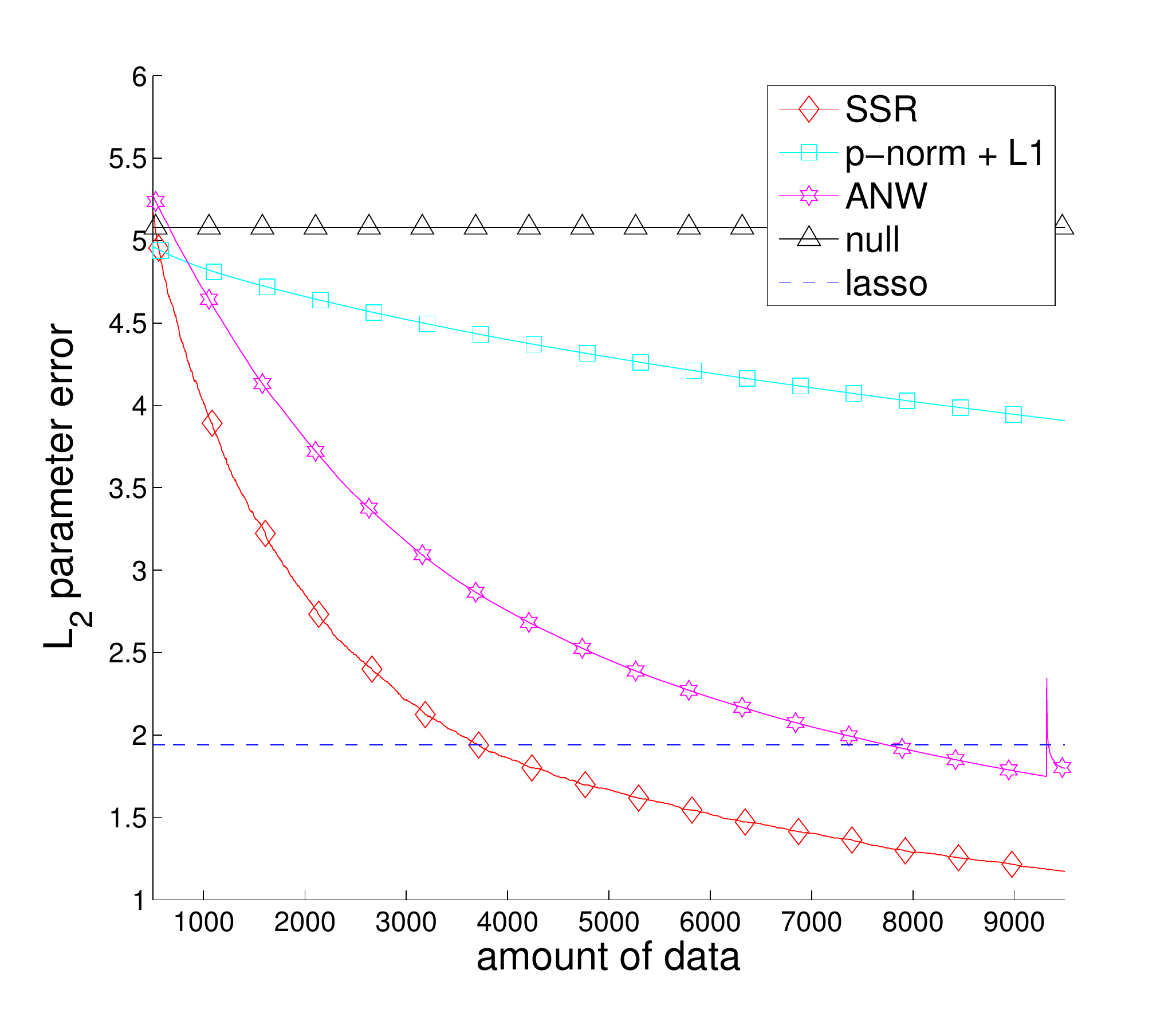} \\
{\begin{sideways}\parbox{\FIGH\textheight}{\centering Linear Reg., Correlated Features}\end{sideways}}&
\includegraphics[height=\FIGH\textheight]{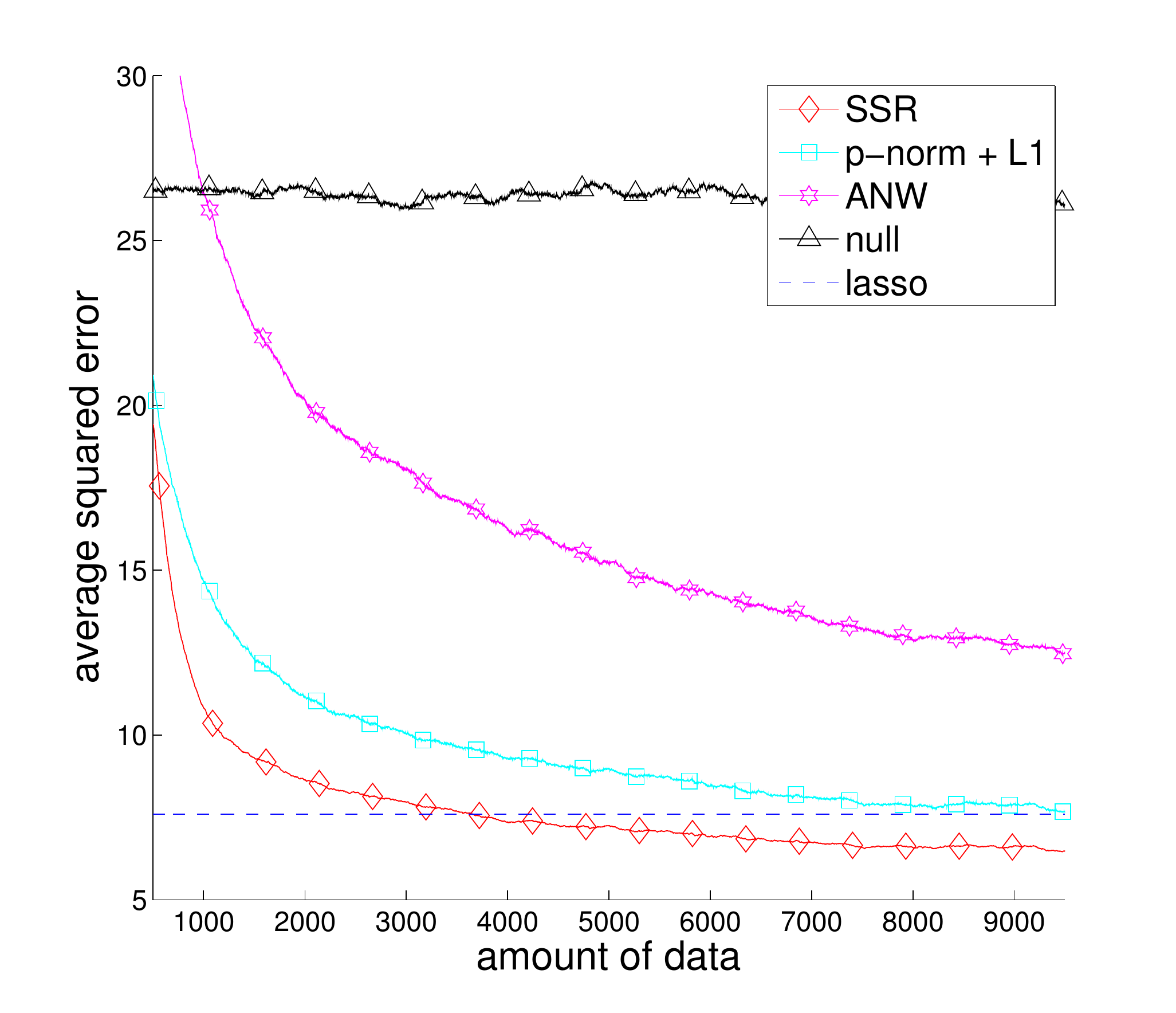} &
\includegraphics[height=\FIGH\textheight]{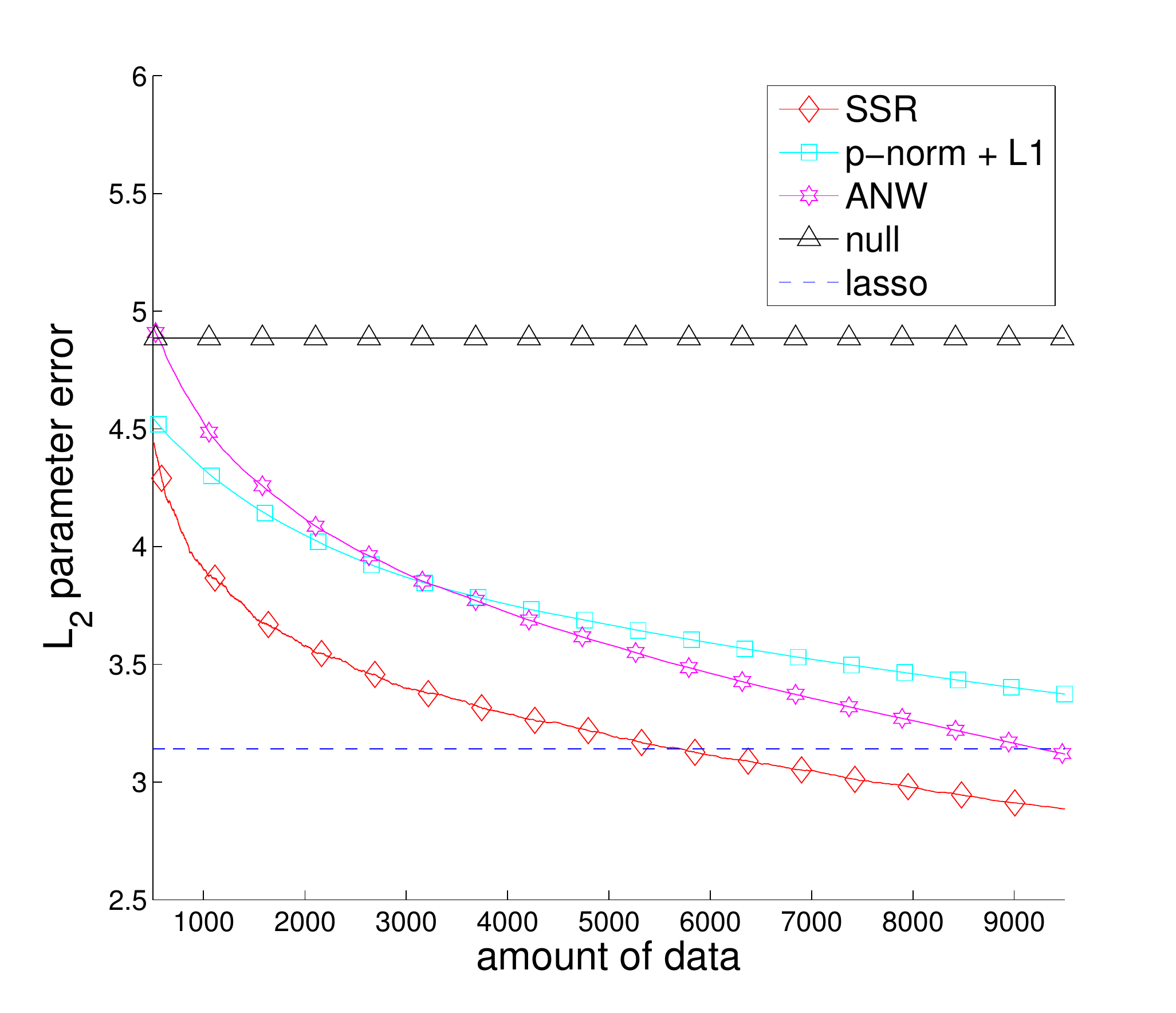} \\
{\begin{sideways}\parbox{\FIGH\textheight}{\centering Logistic Regression}\end{sideways}}&
\includegraphics[height=\FIGH\textheight]{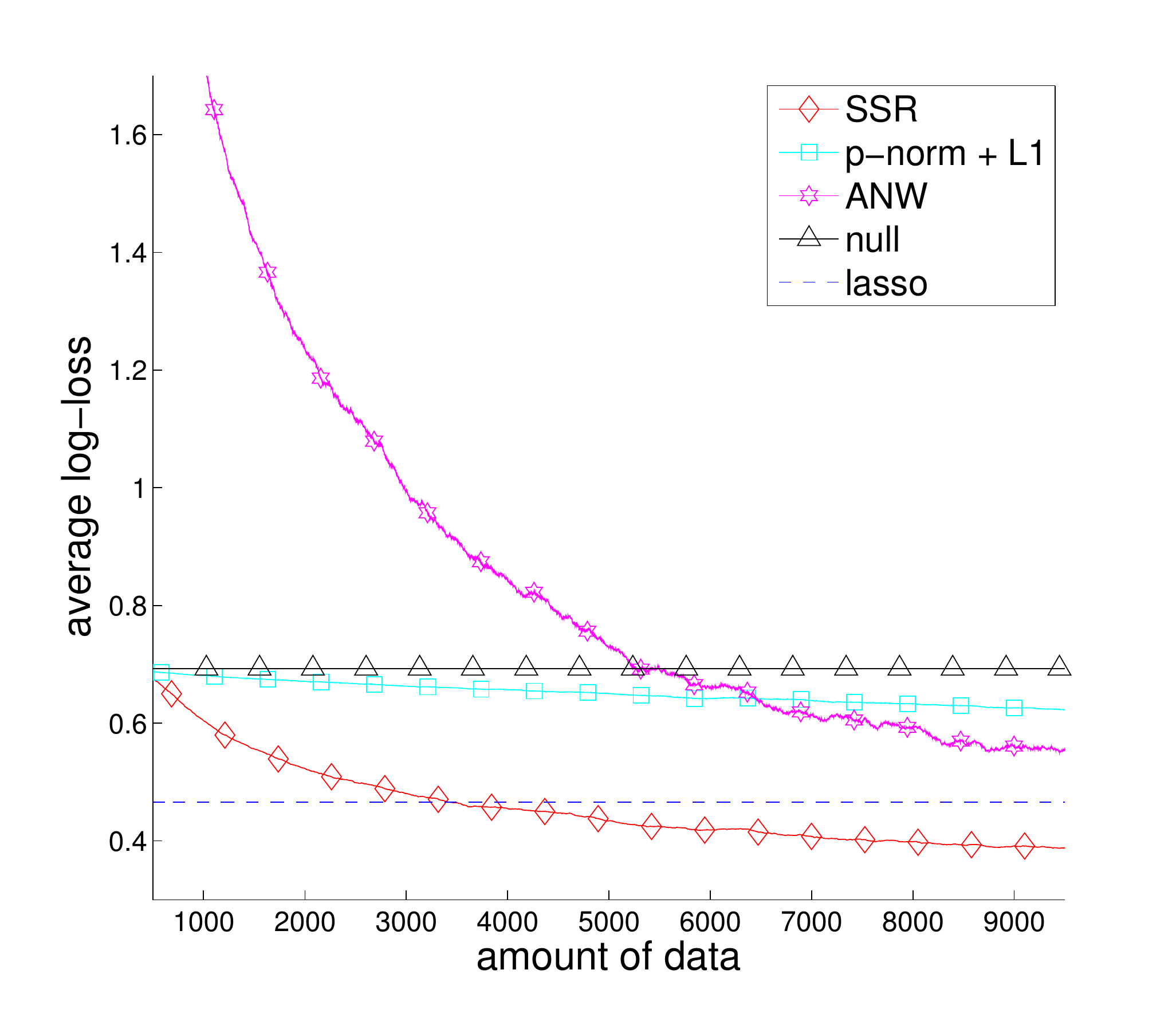} &
\includegraphics[height=\FIGH\textheight]{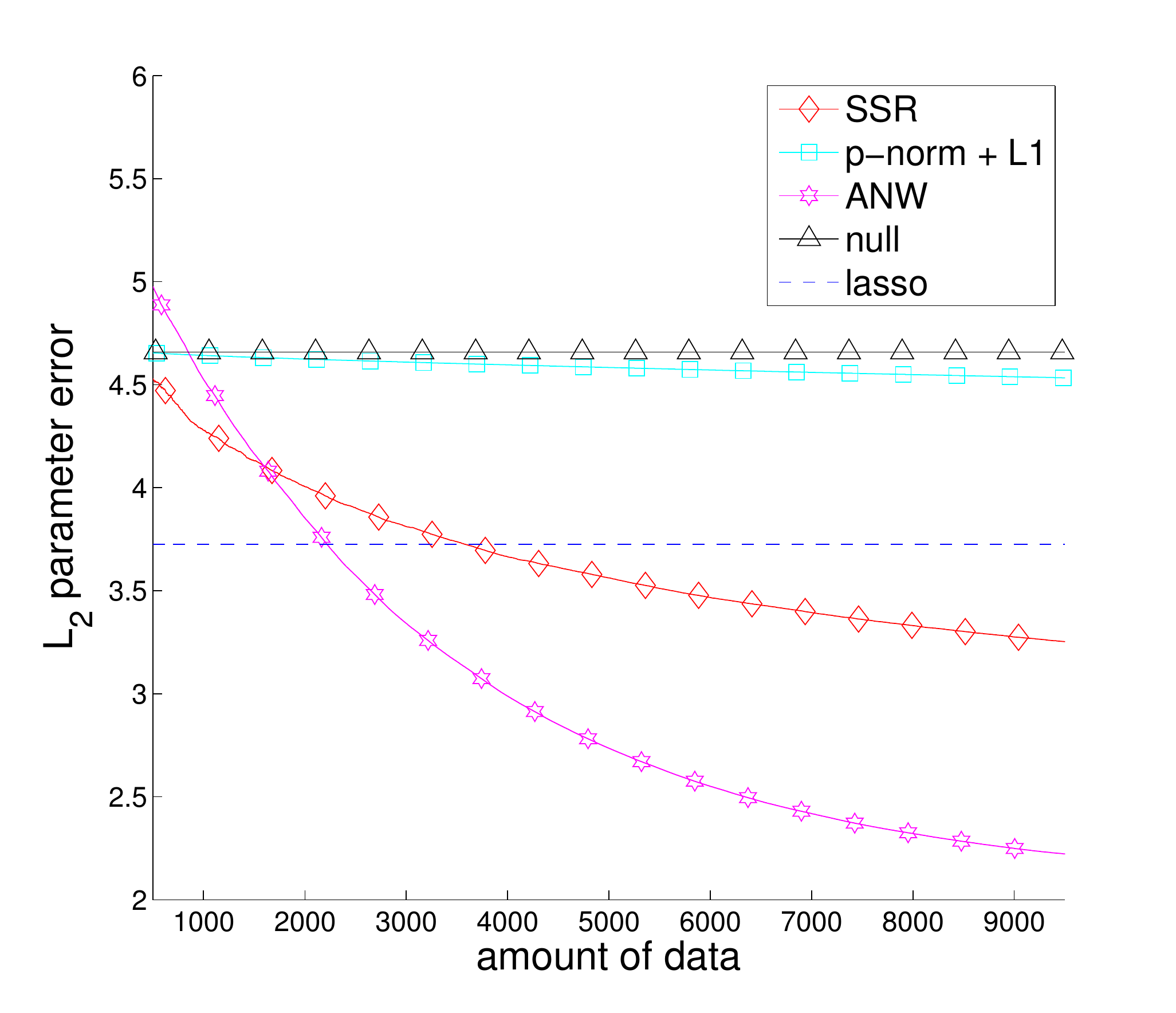}
\end{tabular}}
\caption{Simulation results. The prediction error is in terms of Huberized quadratic loss or logistic loss. We ran each algorithm with $T = 10,000$ training 
examples in total. 
The spike in error for ANW in the first row is because ANW is an epoch-based algorithm, and error tends to increase temporarily at the start of a new epoch.}
\label{fig:simu}
\end{center}
\end{figure}

Figure \ref{fig:simu} compares the performance
of each algorithm, in terms of both prediction error and parameter error. The prediction 
error at time $t$ is $f_t(w_t)$, where $w_t$ depends only on $(x_{1:t-1},y_{1:t-1})$, 
so that prediction error measures actual generalization ability and hence penalizes overfitting. 
Results are aggregated
over 10 realizations of the dataset for a fixed $w^*$. The prediction error is averaged 
over a sliding window consisting of the latest 1,000 examples. In addition,
timing information for all algorithms is given 
in Table~\ref{fig:timing}.

We first compare the online algorithms. Both SSR and ANW converge in squared error at a 
$\frac{1}{T}$ rate, while the $p$-norm algorithm converges at only a $\frac{1}{\sqrt{T}}$ rate. 
This can be seen in most of the plots, where SSR and ANW both outperform the $p$-norm algorithm; 
the exception is in the correlated inputs case, where the $p$-norm algorithm outperforms ANW 
in prediction error by a large margin and is not too much worse than SSR. The reason is that 
the $p$-norm algorithm is highly robust to correlations in the data, while ANW and SSR rely on 
restricted strong convexity and irrepresentability conditions, respectively, which tend to 
degrade as the inputs become more correlated.

We also note that, in comparison to other methods, ANW performs better in terms of
parameter error than prediction error.
The difference is particularly striking for the logistic regression task, where ANW has very poor prediction error 
but very good parameter error (substantially better than all other methods).
The fact that ANW incurs large losses while achieving low parameter error in the
classification example is not contradictory because,
with logistic regression, it is possible to obtain high prediction accuracy 
without recovering the optimal parameters.

Comparison with the lasso fit by \texttt{glmnet}, which we treat as an oracle, yields some interesting
results. Recall that the lasso was only trained using 2,500 training examples, as this was the most data
\texttt{glmnet} could handle before crashing. When the streaming methods have access to 
only 2,500 examples as well, the lasso is beating all of them, 
just as we would expect. However, as we bring in more data, our SSR method
starts to overtake it: in all examples, our method achieves lower prediction error around 4,000 training
examples. This phenomenon emphasizes the fact that, with large datasets, having computationally
efficient algorithms that let us work with more data is desirable.

Finally we note that, in terms of runtime, SSR is by far the fastest method, running 4 to 10 times 
faster than either of the two other algorithms. We emphasize that none of these methods were optimized, so the 
runtime of each method should be taken as a rough indicator rather than an exact measurement of 
efficiency. The bulk of the runtime difference among the online algorithms is due to the fact that both 
ANW and the $p$-norm algorithm require expensive floating point operations like taking $p$-th powers, 
while SSR requires only basic floating point operations like multiplication and addition.

\subsubsection*{Tuning} We selected the tuning parameters
using a single development set of size $1,000$. The tuning parameters
for $p$-norm and ANW are a step size and $L_1$ penalty, and the tuning parameters for 
SSR are the constants $\epsilon$, $\alpha$, and $\lambda$ in Algorithm~\ref{alg:stream}, 
the first two of which control the step size and the last of which controls the $L_1$ penalty.

\subsection{Genomics Data}
\label{sec:application}

The dataset, collected by the Wellcome Trust Case Control Consortium \citep{burton2007genome}, is
a genome-wide association study, comparing $d = 500,568$ single nucleotide polymorphisms (SNPs). The dataset
contains 2,000 cases of type 1 diabetes (T1D), and 1,500 controls,\footnote{The dataset
\citep{burton2007genome} has 3,000 controls, split into 2 sub-populations. We used one of the
two control populations (NBS).}
for a total of $T = 3,500$ data points. We coded each
SNP as 0 if it matches the wild type allele, and as 1 else.

We compared the same methods as before, using a random subset of $500$ data points for tuning 
hyperparameters (since the dataset is already small, we did not create a separate development set). We only 
compute prediction error since the true parameters are unknown. In Figure~\ref{fig:genomics}, we 
plot the prediction error averaged over $40$ random permutations of the data and over a sliding 
window of length $500$.
The results look largely similar to our simulations. As before, SSR outperforms the other streaming
methods, and eventually also beats the lasso oracle once it is able to see enough training data.

\begin{figure}
\centering
\includegraphics[height=0.4\textheight]{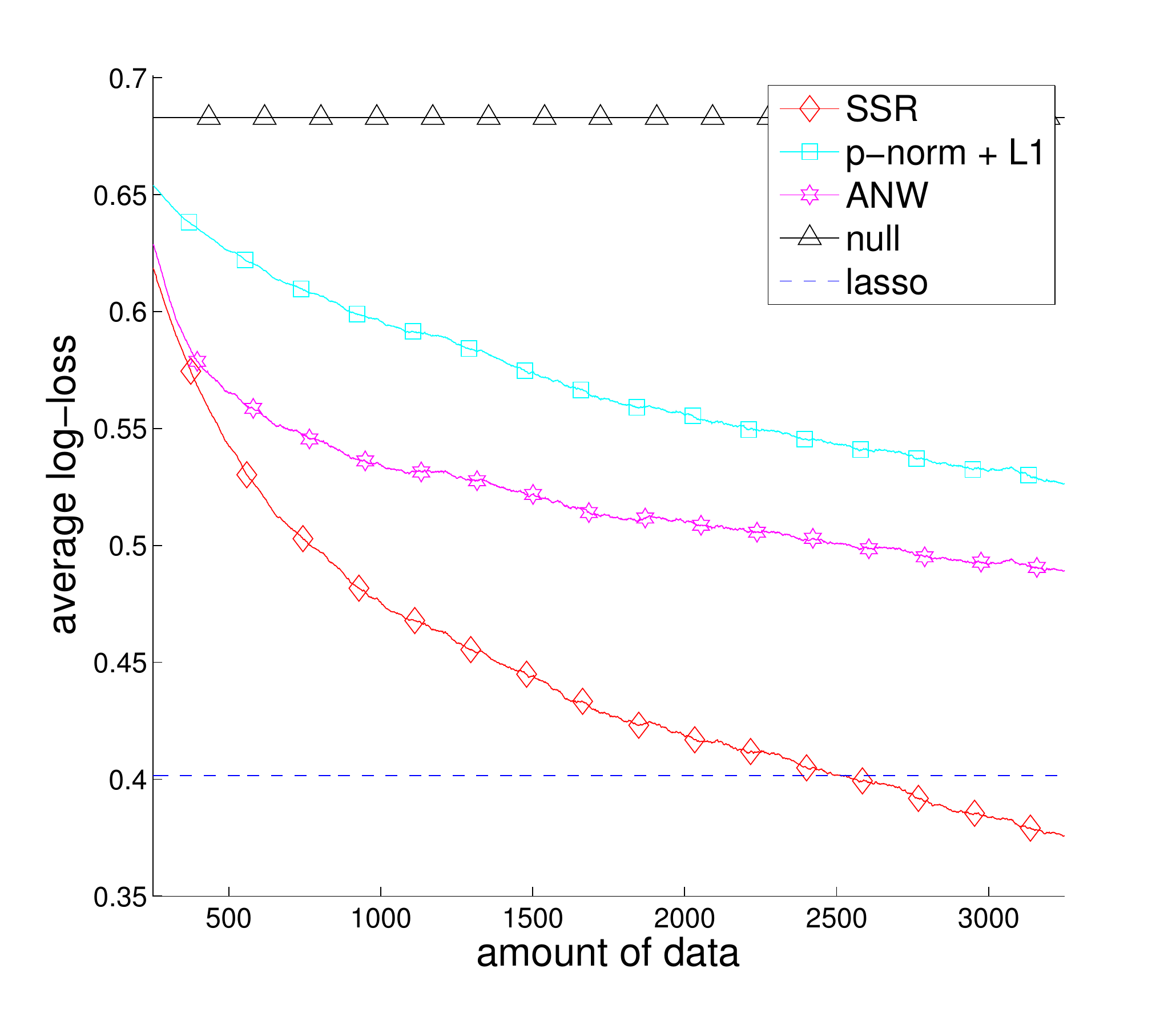}
\caption{Genomics example; logistic loss vs. amount of data.}
\label{fig:genomics}
\end{figure}

\section{Adaptive Mirror Descent with Sparsity Guarantees}
\label{sec:sparse}

We now begin to flesh out the intuition described in Section~\ref{sec:outline}. 
Our first goal is to provide an analogue to the mirror descent bound in Proposition
\ref{prop:adaptive} that takes advantage of sparsity. Intuitively, 
online algorithms with sparse weights $\Norm{w_t}_0 \leq k$ should behave as though they were 
evolving in a $k$-dimensional space instead of a $d$-dimensional space. However, the baseline 
bound \eqref{eq:meta} does not take advantage of this at all: it depends on $\|\nabla f_t(w_t)\|_2^2$, 
which could be as large as $B^2d$.

In this section, we strengthen the adaptive mirror descent bound of \citet{orabona2013general} 
in a way that reflects the effective sparsity of the $w_t$. 
We state our results in the standard adversarial setup. Statistical assumptions will become 
important in order to bound the cost of $\LI$-penalization (Section \ref{sec:statistical}).

Our main result that strengthens the adaptive mirror descent bound is Lemma~\ref{lem:bregman-restricted}, 
which replaces the Bregman divergence term $D_{\psi_t^*}\p{\theta_{t+1} || \theta_t}$ in \eqref{eq:adaptive} with the 
smaller term $D_{\psi_t^*}\p{\theta_{t+1}\sm{[S_t]} || \theta_t\sm{[S_t]}}$, which measures only 
the divergence over a subset $S_t$ of the coordinates.
As before, $\theta_{t+1}\sm{[S_t]}$ denotes the coordinates of $\theta_{t+1}$ that 
belong to $S_t$, with the rest of the coordinates zeroed out. We also let 
$\supp(w_t)$ denote the set of non-zero coordinates of $w_t$. 
Throughout, we defer most proofs to the appendix.

\begin{lemma}[adaptive mirror descent with sparsity]
\label{lem:bregman-restricted}
Suppose that adaptive mirror descent \eqref{eq:mirror} is run with convex regularizers $\psi_t$, and let $S_t$ be a set satisfying:
\begin{enumerate}
\item $\supp(w_t) \subseteq S_t$
\item $\supp(w_{t+1}) \subseteq S_t$
\item For all $w\sm{[S_t]}$, $\psi_t(w\sm{[S_t]},\tilde{w}\sm{[\neg S_t]})$ 
      is minimized at $\tilde{w}\sm{[\neg S_t]} = 0$.
\end{enumerate}
Then, 
\begin{align}
\label{eq:bregman-restricted}
\sum_{t=1}^T &\p{w_t-u}^{\top} \nabla f_t(w_t) 
\\ \notag
&\leq \psi_{T}(u) + \sum_{t=1}^T D_{\psi^*_t}\p{\theta_{t + 1}\sm{[S_t]} || \theta_{t}\sm{[S_t]}} + \sum_{t=1}^T [\psi_{t-1}(w_t) - \psi_t(w_t)].
\end{align}
\end{lemma}

We emphasize that this result does not require any statistical assumptions about the
data-generating process, and relies only on convex optimization machinery. Later, we
will use statistical assumptions to control the size of the active set $S_t$ and thus
bound the right-hand-side of \eqref{eq:bregman-restricted}.

If we apply Lemma~\ref{lem:bregman-restricted} to the choice of $\psi_t$ given 
in \eqref{eq:psi-def}, we get Lemma~\ref{lemm:sparse_amd} below.
The resulting bound is identical to the one in \eqref{cor:adaptive}, except we have replaced 
$\|\nabla f_t(w_t)\|_2^2$ with a term that depends only on an \emph{effective dimension} $k_t$.

\begin{lemma}[decomposition with sparsity]
\label{lemm:sparse_amd}
Let $f_t(\cdot)$ be a sequence of convex loss functions, and let $w_t$ be selected by 
adaptive mirror descent with regularizers \eqref{eq:psi-def}. Then
\begin{align}
\label{eq:sparse_amd}
\sum_{t = 1}^T &\p{w_t-u}^{\top} \nabla f_t \p{w_t} \le \MainRegret + \CostOfSparsity + \Quadratic, \where \\
\label{eq:MainRegret}
\MainRegret & = \frac{\epsilon}{2} \Norm{u}_2^2 + \frac{B^2}{2} \sum_{t=1}^T \frac{k_t}{\epsilon + \eta\, t},
\end{align}
$\MainRegret$ replaces $\MainRegret_0$ in Corollary \ref{cor:adaptive}, and
$\CostOfSparsity$ and $\Quadratic$ are defined in \eqref{eq:CostOfSparsity} and \eqref{eq:Quadratic}.
Here, $k_t = \Abs{S_t}$ is the number of \emph{active features}, and we take 
$S_t = \cup_{s=1}^{t+1} \supp(w_s)$.
\end{lemma}

\paragraph{Example: forcing sparsity}
The statement of Lemma \ref{lemm:sparse_amd} is fairly abstract, and so it can be helpful to elucidate its implications with 
some examples. First, suppose that we determine the $\lambda_t$ sequence in such a way to force the $w_t$ to be $k$-sparse:
\begin{equation}
\label{eq:lambda_force_sparse}
\lambda_{t + 1} = \max \left\{ \lambda_t, \, \left|\theta_t\right|^{(k+1, \, d)} + B  \right\},
\end{equation}
where $|\theta_t|^{(k+1, \, d)}$ denotes the $(k+1)$-st largest (in absolute
magnitude) coordinate of $\theta_t$. Also 
suppose that we set $\eta = 0$ for simplicity.
Then, we can simplify our result to the following:

\begin{corollary}[simplification with sparsity]
\label{coro:force_sparse}
Under the conditions of Lemma \ref{lemm:sparse_amd}, suppose that $\lambda_t$ is 
set using \eqref{eq:lambda_force_sparse}, $\eta = 0$, and that the $f_t(\cdot)$ are 
convex. Then, we obtain the regret bound
\begin{equation}
\label{eq:force_sparse}
\Regret(u)
\leq \frac{\epsilon R^2}{2}
+ \frac{1}{2\epsilon}kB^2T
+ \lambda_TR.
\end{equation}
If we optimize the bound with $\epsilon = \frac{B}{R}\sqrt{kT}$,
then \eqref{eq:force_sparse} is equal to $R\p{B\sqrt{kT} + \lambda_T}$.
\proof
By convexity of $f_t$, we have $f_t(w_t) - f_t(u) \le (w_t - u)^\top \nabla f_t(w_t)$.
Also, since $\eta = 0$, we can actually take $S_t = \supp(w_t) \cap \supp(w_{t+1})$ and 
still satisfy the conditions of Lemma~\ref{lem:bregman-restricted}.
To get the RHS of \eqref{eq:force_sparse} from \eqref{eq:sparse_amd}, we 
use the inequalities $k_t = |S_t| \leq k$, 
$\Norm{u}_2 \leq \Norm{u}_1 \leq R$ and $(\lambda_{t - 1} - \lambda_t)\|w_t\|_1 \leq 0$; this latter 
inequality implies that $\sum_{t=1}^T (\lambda_{t-1} - \lambda_t)(\|w_t\|_1 - \|u\|_1) \leq \lambda_T \|u\|_1$.
\endproof
\end{corollary}
We have shown that stochastic gradient descent can achieve regret
that depends on the sparsity level $k$ rather than the ambient dimension $d$,
as long as the $L_1$ penalty is large enough.
Previous analyses \citep[e.g.,][]{xiao2010dual}
had an analogous regret bound of $R(B\sqrt{dT} + \lambda_T)$, which could be substantially worse
when $d$ is large.

\subsection{Interlude: Sparse Learning with Strongly Convex Loss Functions}
\label{sec:strong_cvx}

In the above section we showed that, when working with generic convex
loss functions $f_t$, we could use our framework to improve a 
$\sqrt{dT}$ factor into $\sqrt{kT}$; in other words, we could 
bound the regret in terms of the effective dimension $k$ rather than 
the ambient dimension $d$. We can thus achieve low regret in high dimensions while using 
an $L_2$-regularizer, as opposed to previous work \citep{shalev2011stochastic} that used an
$L_p$-regularizer with $p = \frac{2\log(d)}{2\log(d)-1}$. 
This fact becomes significant when 
we consider strong convexity properties of our loss functions, where it
is advantageous to use a regularizer with the same strong 
convexity structure as the loss, and where $L_2$-strong convexity of the loss
function is much more common than strong convexity 
in other $L_p$-norms.

In the standard online convex optimization setup, it is well
known \citep{duchi2010composite,hazan2007logarithmic} that 
if the loss functions $f_t$ are strongly convex, we can use faster learning 
rates to get excess risk on the order of $\log T$ rather than $\sqrt{T}$. 
This is because the strong convexity of $f_t$ allows us to remove the 
$\sum_{t=1}^T \|w_t-u\|_2^2$ term from bounds like \eqref{eq:sparse_amd}.

In practice, the loss function $f_t$ is only strongly convex in expectation,
and we will analyze this setting in Section~\ref{sec:expected_cvx}.
But as a warm-up, let us analyze the case where each $f_t$ is actually strongly convex. 
In this case, we can remove the $Q$ term from our regret bound \eqref{eq:sparse_amd} entirely:

\begin{theorem}[decomposition with sparsity and strong convexity]
\label{theo:strong_sparse}
Suppose that we are given a sequence of $\alpha$-strongly convex
losses $f_1, \, \ldots, \, f_T$, and that we run adaptive mirror descent with the
regularizers $\psi_t$ from \eqref{eq:psi-def} with $\eta = \alpha$.
Then, using $\MainRegret$ and $\CostOfSparsity$ from \eqref{eq:sparse_amd}, we have
\begin{align}
\label{eq:strong_sparse}
\Regret(u) =
\sum_{t = 1}^T &\p{f_t\p{w_t} - f_t\p{u}}
\le \MainRegret + \CostOfSparsity. 
\end{align}
\end{theorem}

The key is that with $f_t$ $\alpha$-strongly convex,
we have $f_t\p{w_t} - f_t\p{u} \leq \nabla f_t(w_t)^{\top}\p{w_t-u} - \frac{\alpha}{2}\|w_t-u\|_2^2$,
from which the result follows by invoking \eqref{eq:sparse_amd}.
As a result, we can remove the $\Quadratic$ term while still allowing $\eta > 0$,
which can help reduce $\Omega$.

\paragraph{Example: forcing sparsity}
We can again use the sparsity-forcing schedule $\lambda_t$ from \eqref{eq:lambda_force_sparse} to gain some intuition.

\begin{corollary}[simplification with sparsity and strong convexity]
\label{coro:strong_sparse}
Under the conditions of Theorem \ref{theo:strong_sparse}, suppose that we set $\lambda_t$ 
using \eqref{eq:lambda_force_sparse} and set $\epsilon = 0$. Then
\begin{equation}
\label{eq:coro_strong_sparse}
\Regret(u)
\leq \frac{kB^2}{2 \alpha} \p{1 + \log T} + \lambda_T \Norm{u}_1.
\end{equation}
\end{corollary}

At first glance, it may seem that this result gives us an even better bound than the one 
stated in Theorem~\ref{thm:main-risk}. 
The main term in the bound \eqref{eq:coro_strong_sparse} scales as $\log T$ and
has no explicit dependence on $d$. However, we should not forget the $\lambda_T
\Norm{u}_1$ term required to keep the weights sparse: in general, even if all
but a small 
number of coordinates of $\nabla f_t(w_t)$ are zero-mean random noise,
$\lambda_T$ will need to grow as $\sqrt{T}$ (in fact, $\sqrt{T\log(d)}$) in
order to preserve sparsity. 
This is because an unbiased random walk will still have deviation $\sqrt{T}$
from zero after $T$ steps.
Thus, although we managed to make the main term of the regret bound small,  
the $\lambda_T \Norm{u}_1$ term still looms.
In the absence of strong convexity, having $\lambda_T = O(\sqrt{T})$ would be
acceptable since the first two terms of \eqref{eq:force_sparse} would grow as $\sqrt{T}$ anyway 
in this case,
but since we are after a $\log T$ dependence, we need to work harder.

In the next section, we will show 
that, if we make statistical assumptions and restrict our attention to the minimizer 
$w^*$ of $\Lstar$, the cost of penalization becomes manageable. 
Specifically, we will show that the 
$\sum_{t = 1}^T \p{\lambda_{t - 1} - \lambda_t} \Norm{w_t}_1$ term in \eqref{eq:strong_sparse}
mostly cancels 
out the problematic $\lambda_T \Norm{u}_1$ term when $u = w^*$, and that the remainder scales 
only logarithmically in $T$.

\section{The Cost of Sparsity}
\label{sec:statistical}

In the previous section, we showed how to control the main term of an adaptive mirror 
descent regret bound by exploiting sparsity. In order to achieve sparsity, however, 
we had to impose an $\LI$ penalty which introduces a cost of sparsity term
\eqref{eq:CostOfSparsity}, which is:
\begin{align}
\label{eq:CostOfSparsity2}
\Lambda & \eqdef \sum_{t=1}^T \p{\lambda_{t-1} - \lambda_t}\p{\Norm{w_t}_1 - \Norm{w^*}_1}.
\end{align}
Before, our regret bounds \eqref{eq:strong_sparse}
held against any comparator $u \in \sH$,
but all the results in this section rely on statistical assumptions and thus
will only hold when $u = w^*$, the expected risk minimizer.

In general, we will need $\lambda_T$ to scale as $\sqrt{T \log d}$ to ensure sparsity.
If we use the naive upper bound 
$\Lambda = \lambda_T \Norm{w^*}_1 + \sum_{t = 1}^T \p{\lambda_{t - 1} - \lambda_t} \Norm{w_t}_1
\leq \lambda_T \Norm{w^*}_1$, which holds so long as $\lambda_{t} \ge \lambda_{t-1}$,
we again get regret bounds that grow as $\sqrt{T}$, even under statistical assumptions.
However, we can do better than this naive bound:
we will show that it is possible to substantially cut the cost of sparsity by using 
an $\LI$ penalty that grows steadily in $t$;
in our analysis, we use $\lambda_t = \lambda\sqrt{t+1}$. Using Assumptions (1-3) from 
Section~\ref{sec:main}, we can obtain bounds for $\Lambda$ that grow only logarithmically in $T$:

\begin{lemma}[cost of sparsity]
\label{lemm:l1_cost}
Suppose that Assumptions (1-3) of Section \ref{sec:main} hold, and that $\lambda_t = \lambda\sqrt{t+1}$. 
Then, for any $\delta > 0$, with probability $1 - \delta$,
\begin{align}
\label{eq:l1_cost}
&\CostOfSparsity  \leq \frac{\lambda}{2} \, \sqrt{k \p{1 + \log T}} \\
\notag
&\ \ \ \ \ \ \times \sqrt{\frac{4}{\alpha} \sum_{t = 1}^T \p{f_t\p{w_t\sm{[S]}} - f_t \p{w^*}} +\frac{9 k B^2 \log \p{\log_2(2T)/\delta}}{4\alpha^2}}.
\end{align}
\end{lemma}

Note that Lemma~\ref{lemm:l1_cost} bounds $\Lambda$ in terms of what is essentially 
the square root of $\Regret(w^*)$.
Indeed, if $\supp(w_s) \subseteq S$, so that $w_t\sm{[S]} = w_t$, 
then the sum appearing inside the square-root is exactly $\Regret(w^*)$.
Using this bound, we can provide a recipe for transforming regret bounds for $\LI$-penalized adaptive 
mirror descent algorithms into much stronger excess risk bounds.

\begin{theorem}[cost of sparsity for online prediction error]
\label{theo:l1_recipe}
Under the conditions of Lemma \ref{lemm:l1_cost}, suppose that we have any 
excess risk bound of the form
\begin{equation}
\label{eq:l1_recipe_init}
\Regret(w^*) 
\leq R_T\p{w^*} +  \CostOfSparsity
\end{equation}
for some main term $R_T(w^*) \geq 0$ and $\Lambda$ as defined in 
\eqref{eq:CostOfSparsity2}. 
Then, for regularization schedules of the form $\lambda_t = \lambda\sqrt{t+1}$, 
the following excess risk bound also holds with probability $1 - \delta$ for 
any $\delta > 0$:
\begin{align}
\label{eq:l1_recipe_goal}
&\Regret(w^*) 
 \leq 2 R_T\p{w^*} \\
\notag
&\ \ \ \  + \frac{4k\lambda^2\p{1+\log T}}{3\alpha}
+ \frac{kB^2\log\p{\lg(2T)/\delta}}{2\alpha} + \max\p{0, -\Delta}, \\
\notag
&\ \ \ \ \text{where } \Delta = \sum_{t = 1}^T \p{f_t\p{w_t} - f_t\p{w_t\sm{[S]}}}.
\end{align}
\end{theorem}

Notice that this result does not depend on any form of orthogonal noise
or irrepresentability assumption. Instead, our bound depends implicitly on the assumption
that sparsification improves the performance of our predictor.
Specifically,
$\Delta$
is the excess risk we get from using non-zero weights outside the sparsity set $S$. 
If the non-signal features are pure noise (i.e., independent from the response), then
clearly
$$\EE{f_t(w_t\sm{[S]})} \leq \EE{f_t(w_t)}, $$
and so $\EE{\Delta} \ge 0$ and thus \eqref{eq:l1_recipe_goal} is a strong bound
in the sense that the cost of sparsity grows only as $\log T$.
Conversely,
if there are many good non-sparse models, then $-\Delta$ could potentially be large enough
to render the bound useless.

To use Theorem \ref{theo:l1_recipe} in practice, we will make assumptions
(such as irrepresentability) that
guarantee that $\supp(w_t) \subseteq S$ with high probability for all $t$,
so that $w_t\sm{[S]} = w_t$ and thus $\Delta = 0$.
The following result gives us exactly this guarantee in the case
where the noise features $\neg S$ are orthogonal to the signal $S$ (formalized as Assumption 4),
by letting $\lambda_t$ grow at an appropriate rate.
In Section \ref{sec:irrep}, we relax the orthogonality assumption
to one where the noise features need only be irrepresentable.

\begin{lemma}[support recovery with uncorrelated noise]
\label{lemm:pure_noise}
Suppose that Assumptions 1,  3, and 4 hold.
Then, for any convex functions $\{\zeta_t^{i}\}_{i = 1}^d$, as long as 
$\zeta_t^{i}$ is minimized at $0$ for all $i \not\in S$,
the weights $w_t$ generated by adaptive mirror descent with regularizer
$$ \psi_t\p{w} = \sum_{i = 1}^d \zeta_t^i\p{w^i} +\lambda_t \Norm{w}_1 $$
and
\begin{equation}
\label{eq:pure_noise_lambda}
\lambda_t = c_\delta \sqrt{t} \; \with \; c_\delta = \frac{3B}{2}\sqrt{\log \p{\frac{2d\lg(2T)}{\delta}}}
\end{equation}
will satisfy $\supp(w_t) \subseteq S$ for all $t = 1, \, \dots, \, T$ with probability at least $1 - \delta$.
\end{lemma}

We have thus cleared the main theoretical hurdle identified at the
end of Section \ref{sec:strong_cvx}, by showing that having an $\LI$ penalty that grows
as $\sqrt{t}$ does not necessarily make the regret bound scale with $\sqrt{t}$ also.
Thus, we can now use
Theorem \ref{theo:l1_recipe} in combination with Lemma \ref{lemm:pure_noise} to
get logarithmic bounds on the cost of sparsity $\Lambda$ for strongly convex losses,
as shown below.

\begin{corollary}[synthesis]
\label{coro:strong_cvx_lambda_cost}
Suppose that Assumptions 1, 3, and 4 hold, that we run Algorithm \ref{alg:stream} with
$\varepsilon = 0$ and $\lambda_t = c_\delta \sqrt{t}$ with $c_\delta$ as defined in
\eqref{eq:pure_noise_lambda}. Moreover, suppose that the loss functions $f_t$ are
all $\alpha$-strongly convex for some $\alpha > 0$ and that we set $\eta = \alpha$. Then,
$$ \sum_{t = 1}^T \p{f_t(w_t) - f_t(w^*)} = \oop{\frac{k B^2}{\alpha} \log\p{d \log\p{T}} \log\p{T}}. $$
\proof
This result follows directly by combining Theorem \ref{theo:strong_sparse} with
Theorem \ref{theo:l1_recipe}, while using Lemma \ref{lemm:pure_noise} to control
sparsity.
\endproof
\end{corollary}

\section{Online Learning with Strong Convexity in Expectation}
\label{sec:expected_cvx}

Thus far, we have obtained our desired regret bound of $\oop{k \log(d \log(T)) \log(T)}$,
but assuming that each loss function $f_t$ was $\alpha$-strong convexity
(Corollary \ref{coro:strong_cvx_lambda_cost}).
This strong convexity assumption, however, is unrealistic for many commonly-used loss functions.
For example, in the case of linear regression with $f_t(w) = \frac12 (y_t - w^\top x_t)^2$,
the individual loss functions $f_t$ are not strongly convex. However, we do know that the $f_t$
are strongly convex in expectation as long as the covariance of $x$ is non-singular.  
In this section, we show that
this weaker assumption of strong convexity in expectation is all that is needed to obtain the
same rates as before.

The adaptive mirror descent bounds presented in Sections \ref{sec:main} and
\ref{sec:sparse} all depend on the following inequalities: if the loss function
$f_t(\cdot)$ is convex, then
\begin{equation}
\label{eq:cvx}
f_t\p{w_t} - f_t\p{u} \leq \p{w_t - u}^\top\nabla f_t \p{w_t} ,
\end{equation}
and if $f_t(\cdot)$ is $\alpha$-strongly convex, then
\begin{equation}
\label{eq:cvx_strong}
f_t\p{w_t} - f_t\p{u} \leq \p{w_t - u}^\top\nabla f_t \p{w_t}  - \frac{\alpha}{2}\|w_t-u\|_2^2.
\end{equation}
It turns out that we can use similar arguments even when the losses
$f_t(\cdot)$ are not convex, provided 
that $f_t(\cdot)$ is convex \emph{in expectation}. The following lemma is the key technical device allowing us to do so.
Comparing \eqref{eq:expected_cvx} with \eqref{eq:cvx_strong},
notice that we only lose a factor of $2$ in terms of $\alpha$ and pick up an
additive constant for the high probability guarantee.

\begin{lemma}[online prediction error with expected strong convexity]
\label{lemm:expected_cvx}
Let $f_1, \, \ldots, \, f_T$ be a sequence of (not necessarily convex) loss functions defined over a 
convex region $\sH$ and let $u, \, w_1, \, \ldots, \, w_T \in \sH$. Finally let 
$\sF_0,\sF_1,\ldots$ be a filtration such that:
\begin{enumerate}
\item $w_t$ is $\sF_{t-1}$-measurable, and $u$ is $\sF_0$-measurable,
\item $f_t$ is $\sF_t$-measurable and $\EE{f_t \mid \sF_{t-1}}$ is 
      $\alpha$-strongly convex with respect to some norm $\Norm{\cdot}$, and
\item $f_t$ is almost surely $L$-Lipschitz with respect to $\Norm{\cdot}$ over all of $\sH$.
\end{enumerate}
Then, with probability at least $1-\delta$, we have, for all $T \geq 0$,
\begin{align}
\label{eq:expected_cvx}
\Regret(u)
&\leq
\sum_{t=1}^T \p{\p{w_t - u}^\top\nabla f_t \p{w_t} - \frac{\alpha}{4}\|w_t - u \|^2} +
\frac{8L^2\log(1/\delta)}{\alpha}.
\end{align}
\end{lemma}

We can directly use this lemma to get an extension of the adaptive mirror
descent bound of \citet{orabona2013general} for loss functions that are only
convex in expectation, thus yielding
an analogue of Theorem \ref{theo:strong_sparse},
that only requires expected strong convexity instead of
strong convexity. Note that the above result holds for any fixed $u$, although 
we will always invoke it for $u = w^*$.

\begin{theorem}[simplification with expected strong convexity]
\label{theo:sparse_expected_cvx}
Suppose that the $f_t(\cdot)$ are a sequence of loss functions satisfying Assumptions (1-4), 
and that we run adaptive mirror descent with the regularizers $\psi_t$ from \eqref{eq:psi-def}
and $\varepsilon = 0$. 
Then, assuming that\, $\supp(w_t) \subseteq S$ for all $t$,
we have that for any $\delta > 0$, with probability at least $1 - \delta$,
\begin{align}
\label{eq:sparse_expected_cvx_simple}
\Regret(w^*)
\leq \frac{kB^2}{\alpha}\p{{1 + \log T} + 8\log\p{\frac{1}{\delta}}}
+ \CostOfSparsity,
\end{align}
where $\Lambda$ is the cost of sparsity as defined in \eqref{eq:CostOfSparsity}.
\end{theorem}


We have now assembled all the necessary ingredients to establish our first main result,
namely the excess empirical risk bound for Algorithm \ref{alg:stream} given in Theorem \ref{thm:main-risk},
which states that
\begin{align*}
\Regret(w^*)
 = \oop{\frac{kB^2}{\alpha}\log\p{{d\log(T)}}\log(T)}.
\end{align*}
The proof, provided at the end of Section \ref{sec:apx_expected_cvx}, follows directly from
combining Theorem \ref{theo:l1_recipe}, Lemma \ref{lemm:pure_noise}, and 
Theorem \ref{theo:sparse_expected_cvx}.

We pause here to discuss what we have done so far. At the beginning of Section \ref{sec:sparse},
we set out to provide an excess loss bound for the sparsified stochastic gradient method described in
Algorithm \ref{alg:stream}. The main difficulty was that although $\LI$-induced sparsity enabled us
to control the size of the main term $\MainRegret$ from \eqref{eq:MainRegret}, it induced another
cost-of-sparsity term $\CostOfSparsity$ \eqref{eq:CostOfSparsity} that seemingly grew as $\sqrt{T}$.
However, through the more careful implicit analysis presented in Section \ref{sec:statistical}, we were
able to show that, if $\MainRegret$ satisfies logarithmic bounds in $d$ and $T$, then $\CostOfSparsity$
must also satisfy similar bounds. In parallel, we showed in Section \ref{sec:expected_cvx} how to work with
expected strong convexity instead of actual strong convexity.

The ideas discussed so far, especially in Section \ref{sec:statistical}, comprise the main technical
contributions of this paper. In the remaining pages, we extend the scope of our analysis, by providing
an analogue to Theorem \ref{thm:main-risk} that lets us control parameter error at a quasi-optimal rate,
and by extending our analysis to designs with correlated noise features.

\section{Parameter Estimation using Online-to-Batch Conversion}
\label{sec:batch}

In the previous sections, we focused on bounding the cumulative excess loss made
by our algorithm while streaming over the data, namely $\sum_{t = 1}^T (f_t(w_t) - f_t(w^*))$.
In many cases, however, a statistician may be more interested in estimating the underlying weight
vector $w^*$ than in just obtaining good predictions. In this section, we show how to adapt the
machinery from the previous section for parameter estimation.

The key idea is as follows. Assume that the $f_t$ are \emph{i.i.d.} and recall the expected risk 
is defined as $\law\p{w} = \EE{f_t \p{w}}$. If we know that $\law\p{\cdot}$ is $\alpha$-strongly
convex, we immediately see that, for any $w$,
\begin{equation}
\label{eq:param_bound}
\frac12 \Norm{w - w^*}_2^2 \leq \frac{1}{\alpha} \p{\law\p{w} - \law\p{w^*}}.
\end{equation}
Thus, given a guess $\hw_t$, we can transform any generalization error bound on
$\law\p{\hw_t} - \law\p{w^*}$ into a parameter error bound for $\hw_t$.

The standard way to turn cumulative online loss bounds into generalization bounds
is using ``online-to-batch'' conversion \citep{cesa2004generalization,kakade2009generalization}.
In general, online-to-batch type results tell us that if we start with a bound of the
form\footnote{In Proposition \ref{prop:batch}, we provide one bound of this form that
is useful for our purposes.}
$$ \Regret(w^*) = \sum_{t = 1}^T \p{f_t(w_t) - f_t(w^*)} = \oop{Q(T)}$$
for some function $Q(T)$, then
$$ \law\p{\hat w_T} - \law\p{w^*} = \oop{\frac{Q(T)}{T}}, \; \text{with} \; \hw_T = \frac{1}{T} \sum_{t = 1}^T w_t. $$
The problem with this approach is that, if we applied online-to-batch conversion directly to
Algorithm \ref{alg:stream} and Theorem \ref{thm:main-risk}, we would get a bound of the form
$$ \law\p{\hat w_T} - \law\p{w^*} = \oop{\frac{k\log (d\log(T))}{T} \, \log(T)}, $$
which is loose by a factor $\log T$ with respect to the minimax rate \citep{raskutti2011minimax}.
At a high level, the reason we incur this extra $\log T$ factor is that the required averaging
step $\hw_T = \frac{1}{T} \sum_{t = 1}^T w_t$ gives too much weight to the small-$t$ weights
$w_t$, which may be quite far from $w^*$.

In this section, however, we will show that if we modify Algorithm~\ref{alg:stream} slightly, yielding
Algorithm~\ref{alg:batch}, we can discard the extra $\log T$ factor and
obtain our desired generalization error rate bound of $\oop{{k\log (d\log(T))} / {T}}$.
Besides being of direct interest for parameter estimation,
this technical result will prove to be important in dealing with correlated noise
features under irrepresentability conditions. 

To achieve the desired batch bounds, we modify our algorithm as follows:
\begin{itemize}
\item We replace the loss functions $f_t(w)$ with $\tf_t(w) = t f_t(w)$.
\item We replace the regularizer $\psi_t(w)$ with 
\begin{equation}
\label{eq:batch_reg}
\tpsi_t\p{w} = \frac{1}{2\eta}\p{\sum_{s = 1}^{t} s \Norm{w - w_s}_2^2} + \lambda_t \Norm{w}_1.
\end{equation}
\item We use a correspondingly larger $\LI$ regularizer $\lambda_t = \lambda \cdot t^{3/2}$.
\end{itemize}
Procedurally, this new method yields Algorithm \ref{alg:batch}.
Intuitively, the new algorithm pays more attention to later loss functions and weight vectors
compared to earlier ones.

This construction will allow us to give bounds for
\begin{equation}
\label{eq:batch_goal}
\frac{1}{T} \sum_{t = 1}^T t\p{f_t\p{w_t} - f_t\p{w^*}} \; \text{instead of} \; \sum_{t = 1}^T \p{f_t\p{w_t} - f_t\p{w^*}}.
\end{equation}
It turns out that while the latter is only bounded by $\oop{\log(T)}$, the former
is bounded by $\oop{1}$.
This is useful for proving generalization bounds, as shown by the following online-to-batch
conversion result, the proof of which relies on martingale tail bounds
similar to those developed by \citet{freedman1975tail} and
\citet{kakade2009generalization}.
Note that the weight averaging scheme used in Algorithm
\ref{alg:batch} gives us exactly
$$ \hw_t = \frac{2}{t(t+1)}\sum_{s=1}^t s \, w_s; $$
this equality can be verified by induction.

\begin{proposition}[online-to-batch conversion]
\label{prop:batch}
Suppose that, for any $\delta > 0$, with probability $1 - \delta$,
$$ \sum_{s = 1}^t s\p{f_s\p{w_s} - f_s\p{w^*}} \leq R_\delta \, t $$ 
for all $t \leq T$,
and that each $f_t$ is $L$-Lipschitz and $\alpha$-strongly convex over $\sH$. 
Then, with probability at least $1 - 2\delta$,
\begin{align}
\label{eq:conversion}
\law\p{\hat w_t} - \law\p{w^*} 
&\leq \frac{4R_\delta}{t} + \frac{9\log\p{\lg(2T^3)/\delta}L^2}{\alpha t}
\end{align}
for all $t \leq T$.
\end{proposition}

Given these ideas, we can mimic our analysis from Sections \ref{sec:sparse}, \ref{sec:statistical}
and \ref{sec:expected_cvx} to provide bounds of the form \eqref{eq:batch_goal} for
Algorithm \ref{alg:batch}. Combined with Proposition \ref{prop:batch}, this will
result in the desired generalization bound. For conciseness, we defer this argument
to the Appendix, and only present the final bound below.

\begin{theorem}[expected prediction error with uncorrelated noise]
\label{theo:main_batch}
Suppose that we run
Algorithm \ref{alg:batch}
with $\eta = \alpha/2$ and
$$\lambda = \frac{3B}{2}\sqrt{\log \p{\frac{2d\lg(2T^3)}{\delta}}}, $$ 
and that Assumptions 1-4 hold. 
Then, with probability $1-3\delta$, for all $t \leq T$ we have 
\begin{equation}
\label{eq:main_batch}
\Lstar\p{\hat w_t} - \Lstar(w^*) \leq \frac{269B^2k\log\p{2d\lg(2T^3)/\delta}}{\alpha t},
\end{equation}
and hence
\begin{equation}
\label{eq:main_batch_param}
\frac{1}{2}\Norm{\hat w_t - w^*}_2^2 \leq \frac{269B^2k\log\p{2d\lg(2T^3)/\delta}}{\alpha^2 t}.
\end{equation}
\end{theorem}

With this result in hand, we can obtain
Theorem \ref{thm:main-param} as a direct corollary of Theorem \ref{theo:main_batch}.
Thus, as desired, we have obtained a parameter
error bound that decays as $\oop{k \log\p{d \log (T)} / T}$. As discussed earlier, this is
minimax optimal \citep{raskutti2011minimax} as long as the problem is not extremely low-dimensional (the bound
becomes loose only if $T \gg \exp(d)$).

\section{Streaming Sparse Regression with Irrepresentable Features}
\label{sec:irrep}

Finally, we end our analysis by re-visiting probably the most problematic of our
original 4 assumptions from Section \ref{sec:setup}, namely that the gradients
corresponding to noise features are all mean zero for any weight vector $w$
with support in $S$. Here, we show that this assumption is in fact not needed
in its full strength. In particular, in the case of streaming linear regression, a weaker
irrepresentability condition (Assumption \ref{assu:irrep_noise}) is sufficient to guarantee
good performance of our algorithm.

In our original analysis, mean-zero gradients (Assumption 4)
allowed us to guarantee that our $\LI$ penalization scheme would in fact result in sparse
weights, as in Lemma \ref{lemm:pure_noise}. Below, we provide an analogous
result in the case of irrepresentable noise features.

\begin{lemma}[support recovery with irrepresentability]
\label{lem:irrep_lambda}
Suppose that Assumptions 1-3 and 5 hold, and that we run
Algorithm \ref{alg:batch} with $\eta = \alpha/2$ and
\begin{equation}
\label{eq:irrep_lambda}
\lambda \geq \sqrt{\frac{228B^2\log\p{{2d\lg(2T^3)} \big/ {\delta}}}{1-24\rho^2}}.
\end{equation}
Then, with probability at least $1 - 4\delta$, $\supp(w_t) \subseteq S$ for all
$t = 1, \, \dots, \, T$.
\end{lemma}

Thanks to this sparsity guarantee, we can use similar machinery as in Section \ref{sec:batch}
to bound the generalization error of the output of Algorithm \ref{alg:batch}. Again, just like
in Section \ref{sec:batch}, Theorem \ref{theo:irrep_main}
is a direct consequence of the result below thanks to the $\alpha$-strong convexity of $\law$
and \eqref{eq:param_bound}.

\begin{theorem}[expected prediction error with irrepresentability]
\label{theo:main_irrep}
Let $w_t$ be the weights generated by Algorithm \ref{alg:batch}.
Then, under the conditions of Lemma \ref{lem:irrep_lambda},
with $\hat w_T = \frac{2}{T(T + 1)} \sum_{t = 1}^T t \, w_t$, we have, with probability $1-5\delta$, $\supp(\hat{w_T}) \subseteq S$ and
\begin{align}
\law\p{\hat w_T} - \law\p{w^*} &= \oop{\frac{1}{1-24\rho^2} \, \frac{kB^2\log\p{d\log(T)}}{\alpha T}}, \\
\Norm{\hat w_T - w^*}_2^2 &= \oop{\frac{1}{1-24\rho^2} \, \frac{kB^2\log\p{d\log(T)}}{\alpha^2 T}}.
\end{align}
\end{theorem}


\section{Discussion}
\label{sec:discussion}

In this work, we have developed an efficient algorithm for solving sparse regression 
problems in the streaming setting,
and have shown that it can achieve optimal 
rates of convergence in both prediction and parameter error.
To recap our theoretical contributions: we have shown that online algorithms with 
sparse iterates enjoy better convergence (obtaining a dependence on $k$ rather than $d$); that regularization schedules increasing 
at a $\sqrt{t}$ rate can enjoy very low excess risk under statistical assumptions; 
and that functions that are only strongly convex in expectation can still yield $\log T$
error rather than $\sqrt{T}$.
Together, these show that a 
natural streaming analogue of the lasso achieves convergence at the same rate as 
the lasso itself, similarly to how stochastic gradient descent achieves the same 
rate as batch linear regression.

This work generates several questions. First, can we
weaken the irrepresentability assumption, or more ambitiously, replace it 
with a restricted isometry condition?
This latter goal would require analyzing 
the algorithm in regimes where the support is not recovered, since the restricted 
isometry property is not enough to guarantee support recovery even in a minimax 
batch setting. Another interesting question is whether we can reduce memory usage 
even further --- currently, we use $O(d)$ memory, but one could imagine using 
only $O(k\log(d))$ memory; after all $w^*$ takes only $O(k\log(d))$ memory to store.

Finally, we see this work as one part
of the broader goal of designing computationally-oriented statistical
procedures, which undoubtedly will become increasingly important
in an era when high volumes of streaming data is the norm.
By leveraging online convex optimization techniques, we can analyze specific procedures,
whose computational properties are favorable by construction.
By using statistical thinking, we can obtain much stronger results
compared to purely optimization-based analyses.
We believe that the combination of the two holds general promise,
which can be used to examine other statistical problems in a new computational light.

\bibliographystyle{plainnat}
\bibliography{../references}

\begin{thebibliography}{63}
\providecommand{\natexlab}[1]{#1}
\providecommand{\url}[1]{\texttt{#1}}
\expandafter\ifx\csname urlstyle\endcsname\relax
  \providecommand{\doi}[1]{doi: #1}\else
  \providecommand{\doi}{doi: \begingroup \urlstyle{rm}\Url}\fi

\bibitem[Adelman-McCarthy et~al.(2008)Adelman-McCarthy, Ag{\"u}eros, Allam,
  Prieto, Anderson, Anderson, Annis, Bahcall, Bailer-Jones, Baldry,
  et~al.]{adelman2008sixth}
Jennifer~K Adelman-McCarthy, Marcel~A Ag{\"u}eros, Sahar~S Allam,
  Carlos~Allende Prieto, Kurt~SJ Anderson, Scott~F Anderson, James Annis,
  Neta~A Bahcall, CAL Bailer-Jones, Ivan~K Baldry, et~al.
\newblock The sixth data release of the sloan digital sky survey.
\newblock \emph{The Astrophysical Journal Supplement Series}, 175\penalty0
  (2):\penalty0 297, 2008.

\bibitem[Agarwal et~al.(2012)Agarwal, Negahban, and
  Wainwright]{agarwal2012stochastic}
Alekh Agarwal, Sahand Negahban, and Martin~J Wainwright.
\newblock Stochastic optimization and sparse statistical recovery: Optimal
  algorithms for high dimensions.
\newblock In \emph{Advances in Neural Information Processing Systems}, pages
  1538--1546, 2012.

\bibitem[Alaoui and Mahoney(2014)]{alaoui2014fast}
Ahmed~El Alaoui and Michael~W Mahoney.
\newblock Fast randomized kernel methods with statistical guarantees.
\newblock \emph{arXiv preprint arXiv:1411.0306}, 2014.

\bibitem[Alon et~al.(1996)Alon, Matias, and Szegedy]{alon1996space}
Noga Alon, Yossi Matias, and Mario Szegedy.
\newblock The space complexity of approximating the frequency moments.
\newblock In \emph{Proceedings of the twenty-eighth annual ACM symposium on
  Theory of computing}, pages 20--29. ACM, 1996.

\bibitem[Bache and Lichman(2013)]{bache2013uci}
Kevin Bache and Moshe Lichman.
\newblock {UCI} machine learning repository, 2013.
\newblock URL \url{http://archive.ics.uci.edu/ml}.

\bibitem[Battams(2014)]{battams2014stream}
Karl Battams.
\newblock Stream processing for solar physics: Applications and implications
  for big solar data.
\newblock \emph{arXiv preprint arXiv:1409.8166}, 2014.

\bibitem[Beck and Teboulle(2009)]{beck2009fast}
Amir Beck and Marc Teboulle.
\newblock A fast iterative shrinkage-thresholding algorithm for linear inverse
  problems.
\newblock \emph{SIAM Journal on Imaging Sciences}, 2\penalty0 (1):\penalty0
  183--202, 2009.

\bibitem[Bickel et~al.(2009)Bickel, Ritov, and
  Tsybakov]{bickel2009simultaneous}
Peter~J Bickel, Ya'acov Ritov, and Alexandre~B Tsybakov.
\newblock Simultaneous analysis of lasso and dantzig selector.
\newblock \emph{The Annals of Statistics}, pages 1705--1732, 2009.

\bibitem[Bottou(1998)]{bottou1998online}
L\'{e}on Bottou.
\newblock Online algorithms and stochastic approximations.
\newblock In David Saad, editor, \emph{Online Learning and Neural Networks}.
  Cambridge University Press, Cambridge, UK, 1998.
\newblock URL \url{http://leon.bottou.org/papers/bottou-98x}.
\newblock revised, oct 2012.

\bibitem[Burton et~al.(2007)Burton, Clayton, Cardon, Craddock, Deloukas,
  Duncanson, Kwiatkowski, McCarthy, Ouwehand, Samani, et~al.]{burton2007genome}
Paul~R Burton, David~G Clayton, Lon~R Cardon, Nick Craddock, Panos Deloukas,
  Audrey Duncanson, Dominic~P Kwiatkowski, Mark~I McCarthy, Willem~H Ouwehand,
  Nilesh~J Samani, et~al.
\newblock Genome-wide association study of 14,000 cases of seven common
  diseases and 3,000 shared controls.
\newblock \emph{Nature}, 447\penalty0 (7145):\penalty0 661--678, 2007.

\bibitem[Cand{\`e}s and Tao(2007)]{candes2007dantzig}
Emmanuel Cand{\`e}s and Terence Tao.
\newblock The {D}antzig selector: Statistical estimation when p is much larger
  than n.
\newblock \emph{The Annals of Statistics}, pages 2313--2351, 2007.

\bibitem[Cand{\`e}s and Plan(2009)]{candes2009near}
Emmanuel~J Cand{\`e}s and Yaniv Plan.
\newblock Near-ideal model selection by $l$1 minimization.
\newblock \emph{The Annals of Statistics}, 37\penalty0 (5A):\penalty0
  2145--2177, 2009.

\bibitem[Cesa-Bianchi et~al.(2004)Cesa-Bianchi, Conconi, and
  Gentile]{cesa2004generalization}
Nicolo Cesa-Bianchi, Alex Conconi, and Claudio Gentile.
\newblock On the generalization ability of on-line learning algorithms.
\newblock \emph{Information Theory, IEEE Transactions on}, 50\penalty0
  (9):\penalty0 2050--2057, 2004.

\bibitem[Chen et~al.(1998)Chen, Donoho, and Saunders]{chen1998atomic}
Scott~Shaobing Chen, David~L Donoho, and Michael~A Saunders.
\newblock Atomic decomposition by basis pursuit.
\newblock \emph{SIAM Journal on Scientific Computing}, 20\penalty0
  (1):\penalty0 33--61, 1998.

\bibitem[Cover(1969)]{cover1969hypothesis}
Thomas~M Cover.
\newblock Hypothesis testing with finite statistics.
\newblock \emph{The Annals of Mathematical Statistics}, pages 828--835, 1969.

\bibitem[Crammer et~al.(2006)Crammer, Dekel, Keshet, Shalev-Shwartz, and
  Singer]{crammer2006online}
Koby Crammer, Ofer Dekel, Joseph Keshet, Shai Shalev-Shwartz, and Yoram Singer.
\newblock Online passive-aggressive algorithms.
\newblock \emph{The Journal of Machine Learning Research}, 7:\penalty0
  551--585, 2006.

\bibitem[Duchi et~al.(2010)Duchi, Shalev-Shwartz, Singer, and
  Tewari]{duchi2010composite}
John Duchi, Shai Shalev-Shwartz, Yoram Singer, and Ambuj Tewari.
\newblock Composite objective mirror descent.
\newblock In \emph{Conference on Learning Theory}, 2010.

\bibitem[Efron et~al.(2004)Efron, Hastie, Johnstone, Tibshirani,
  et~al.]{efron2004least}
Bradley Efron, Trevor Hastie, Iain Johnstone, Robert Tibshirani, et~al.
\newblock Least angle regression.
\newblock \emph{The Annals of statistics}, 32\penalty0 (2):\penalty0 407--499,
  2004.

\bibitem[El~Ghaoui et~al.(2010)El~Ghaoui, Viallon, and Rabbani]{el2010safe}
Laurent El~Ghaoui, Vivian Viallon, and Tarek Rabbani.
\newblock Safe feature elimination in sparse supervised learning.
\newblock \emph{CoRR}, 2010.

\bibitem[Fithian and Hastie(2014)]{fithian2014local}
William Fithian and Trevor Hastie.
\newblock Local case-control sampling: Efficient subsampling in imbalanced data
  sets.
\newblock \emph{The Annals of Statistics}, 42\penalty0 (5):\penalty0
  1693--1724, 2014.
\newblock \doi{10.1214/14-AOS1220}.

\bibitem[Flajolet and Nigel~Martin(1985)]{flajolet1985probabilistic}
Philippe Flajolet and G~Nigel~Martin.
\newblock Probabilistic counting algorithms for data base applications.
\newblock \emph{Journal of computer and system sciences}, 31\penalty0
  (2):\penalty0 182--209, 1985.

\bibitem[Freedman(1975)]{freedman1975tail}
David~A Freedman.
\newblock On tail probabilities for martingales.
\newblock \emph{the Annals of Probability}, pages 100--118, 1975.

\bibitem[Friedman et~al.(2010)Friedman, Hastie, and
  Tibshirani]{friedman2010regularization}
Jerome Friedman, Trevor Hastie, and Rob Tibshirani.
\newblock Regularization paths for generalized linear models via coordinate
  descent.
\newblock \emph{Journal of statistical software}, 33\penalty0 (1):\penalty0 1,
  2010.

\bibitem[Garrigues and {El Ghaoui}(2009)]{garrigues2009homotopy}
Pierre Garrigues and Laurent {El Ghaoui}.
\newblock An homotopy algorithm for the lasso with online observations.
\newblock In \emph{Advances in neural information processing systems}, pages
  489--496, 2009.

\bibitem[Gentile(2003)]{gentile2003robustness}
Claudio Gentile.
\newblock The robustness of the p-norm algorithms.
\newblock \emph{Machine Learning}, 53\penalty0 (3):\penalty0 265--299, 2003.

\bibitem[Gerchinovitz(2013)]{gerchinovitz2013sparsity}
S{\'e}bastien Gerchinovitz.
\newblock Sparsity regret bounds for individual sequences in online linear
  regression.
\newblock \emph{The Journal of Machine Learning Research}, 14\penalty0
  (1):\penalty0 729--769, 2013.

\bibitem[Hastie et~al.(2009)Hastie, Tibshirani, and
  Friedman]{hastie2009elements}
Trevor Hastie, Robert Tibshirani, and Jerome Friedman.
\newblock \emph{The Elements of Statistical Learning}.
\newblock New York: Springer, 2009.

\bibitem[Hazan et~al.(2007{\natexlab{a}})Hazan, Agarwal, and
  Kale]{hazan2007logarithmic}
Elad Hazan, Amit Agarwal, and Satyen Kale.
\newblock Logarithmic regret algorithms for online convex optimization.
\newblock \emph{Machine Learning}, 69\penalty0 (2-3):\penalty0 169--192,
  2007{\natexlab{a}}.

\bibitem[Hazan et~al.(2007{\natexlab{b}})Hazan, Rakhlin, and
  Bartlett]{hazan2007adaptive}
Elad Hazan, Alexander Rakhlin, and Peter~L Bartlett.
\newblock Adaptive online gradient descent.
\newblock In \emph{Advances in Neural Information Processing Systems}, pages
  65--72, 2007{\natexlab{b}}.

\bibitem[Hellman and Cover(1970)]{hellman1970learning}
Martin~E Hellman and Thomas~M Cover.
\newblock Learning with finite memory.
\newblock \emph{The Annals of Mathematical Statistics}, pages 765--782, 1970.

\bibitem[Kakade and Tewari(2009)]{kakade2009generalization}
Sham~M Kakade and Ambuj Tewari.
\newblock On the generalization ability of online strongly convex programming
  algorithms.
\newblock In \emph{Advances in Neural Information Processing Systems}, pages
  801--808, 2009.

\bibitem[Kivinen and Warmuth(1997)]{kivinen1997}
Jyrki Kivinen and Manfred~K Warmuth.
\newblock Exponentiated gradient versus gradient descent for linear predictors.
\newblock \emph{Information and Computation}, 132\penalty0 (1):\penalty0 1--63,
  1997.

\bibitem[Langford et~al.(2009)Langford, Li, and Zhang]{langford2009sparse}
John Langford, Lihong Li, and Tong Zhang.
\newblock Sparse online learning via truncated gradient.
\newblock \emph{Journal of Machine Learning Research}, 10\penalty0
  (777-801):\penalty0 65, 2009.

\bibitem[Littlesttone and Warmuth(1989)]{littlestone1989wm}
Nick Littlesttone and Manfred~K Warmuth.
\newblock The weighted majority algorithm.
\newblock In \emph{Foundations of Computer Science, 30th Annual Symposium on},
  pages 256--261. IEEE, 1989.

\bibitem[McMahan(2011)]{mcmahan2011follow}
H~Brendan McMahan.
\newblock Follow-the-regularized-leader and mirror descent: Equivalence
  theorems and l1 regularization.
\newblock In \emph{International Conference on Artificial Intelligence and
  Statistics}, pages 525--533, 2011.

\bibitem[McMahan et~al.(2013)McMahan, Holt, Sculley, Young, Ebner, Grady, Nie,
  Phillips, Davydov, Golovin, et~al.]{mcmahan2013ad}
H~Brendan McMahan, Gary Holt, D~Sculley, Michael Young, Dietmar Ebner, Julian
  Grady, Lan Nie, Todd Phillips, Eugene Davydov, Daniel Golovin, et~al.
\newblock Ad click prediction: a view from the trenches.
\newblock In \emph{Proceedings of the International Conference on Knowledge
  Discovery and Data Mining}, 2013.

\bibitem[Meinshausen and B{\"u}hlmann(2006)]{meinshausen2006high}
Nicolai Meinshausen and Peter B{\"u}hlmann.
\newblock High-dimensional graphs and variable selection with the lasso.
\newblock \emph{The Annals of Statistics}, pages 1436--1462, 2006.

\bibitem[Meinshausen and Yu(2009)]{meinshausen2009lasso}
Nicolai Meinshausen and Bin Yu.
\newblock Lasso-type recovery of sparse representations for high-dimensional
  data.
\newblock \emph{The Annals of Statistics}, 37\penalty0 (1):\penalty0 246--270,
  2009.

\bibitem[Munro and Paterson(1980)]{munro1980selection}
J~Ian Munro and Mike~S Paterson.
\newblock Selection and sorting with limited storage.
\newblock \emph{Theoretical computer science}, 12\penalty0 (3):\penalty0
  315--323, 1980.

\bibitem[Nemirovsky and Yudin(1983)]{nemirovsky1983problem}
Arkadi{\u\i}~S Nemirovsky and David~B Yudin.
\newblock \emph{Problem complexity and method efficiency in optimization}.
\newblock Wiley, New York, 1983.

\bibitem[Orabona et~al.(2013)Orabona, Crammer, and
  Cesa-Bianchi]{orabona2013general}
Francesco Orabona, Koby Crammer, and Nicolo Cesa-Bianchi.
\newblock A generalized online mirror descent with applications to
  classification and regression.
\newblock \emph{arXiv preprint arXiv:1304.2994}, 2013.

\bibitem[Osborne et~al.(2012)Osborne, Roberts, Rogers, and
  Jennings]{osborne2012real}
Michael~A Osborne, Stephen~J Roberts, Alex Rogers, and Nicholas~R Jennings.
\newblock Real-time information processing of environmental sensor network data
  using bayesian gaussian processes.
\newblock \emph{ACM Transactions on Sensor Networks (TOSN)}, 9\penalty0
  (1):\penalty0 1, 2012.

\bibitem[Qian et~al.(2013)Qian, Hastie, Friedman, Tibshirani, and
  Simon]{qian2013glmnet}
J~Qian, T~Hastie, J~Friedman, R~Tibshirani, and N~Simon.
\newblock Glmnet for matlab, 2013.
\newblock URL \url{http://www.stanford.edu/~hastie/glmnet_matlab/}.

\bibitem[Raskutti and Mahoney(2014)]{raskutti2014statistical}
Garvesh Raskutti and Michael Mahoney.
\newblock A statistical perspective on randomized sketching for ordinary
  least-squares.
\newblock \emph{arXiv preprint arXiv:1406.5986}, 2014.

\bibitem[Raskutti et~al.(2010)Raskutti, Wainwright, and
  Yu]{raskutti2010restricted}
Garvesh Raskutti, Martin~J Wainwright, and Bin Yu.
\newblock Restricted eigenvalue properties for correlated gaussian designs.
\newblock \emph{The Journal of Machine Learning Research}, 11:\penalty0
  2241--2259, 2010.

\bibitem[Raskutti et~al.(2011)Raskutti, Wainwright, and
  Yu]{raskutti2011minimax}
Garvesh Raskutti, Martin~J Wainwright, and Bin Yu.
\newblock Minimax rates of estimation for high-dimensional linear regression
  over $l_q$-balls.
\newblock \emph{Information Theory, IEEE Transactions on}, 57\penalty0
  (10):\penalty0 6976--6994, 2011.

\bibitem[Robbins and Monro(1951)]{robbins1951stochastic}
Herbert Robbins and Sutton Monro.
\newblock A stochastic approximation method.
\newblock \emph{The annals of mathematical statistics}, 22\penalty0
  (3):\penalty0 400--407, 1951.

\bibitem[Robbins and Siegmund(1971)]{robbins1971convergence}
Herbert Robbins and David Siegmund.
\newblock A convergence theorem for non negative almost supermartingales and
  some applications.
\newblock In Jagdish~S. Rustagi, editor, \emph{Optimizing Methods in
  Statistics}. Academic Press, 1971.

\bibitem[Shalev-Shwartz(2011)]{shalev2011}
Shai Shalev-Shwartz.
\newblock Online learning and online convex optimization.
\newblock \emph{Foundations and Trends in Machine Learning}, 4\penalty0
  (2):\penalty0 107--194, 2011.

\bibitem[Shalev-Shwartz and Singer(2007)]{shalev2007primal}
Shai Shalev-Shwartz and Yoram Singer.
\newblock A primal-dual perspective of online learning algorithms.
\newblock \emph{Machine Learning}, 69\penalty0 (2-3):\penalty0 115--142, 2007.

\bibitem[Shalev-Shwartz and Tewari(2011)]{shalev2011stochastic}
Shai Shalev-Shwartz and Ambuj Tewari.
\newblock Stochastic methods for {L1}-regularized loss minimization.
\newblock \emph{The Journal of Machine Learning Research}, 12:\penalty0
  1865--1892, 2011.

\bibitem[Shalev-Shwartz et~al.(2010)Shalev-Shwartz, Srebro, and
  Zhang]{shalev2010trading}
Shai Shalev-Shwartz, Nathan Srebro, and Tong Zhang.
\newblock Trading accuracy for sparsity in optimization problems with sparsity
  constraints.
\newblock \emph{SIAM Journal on Optimization}, 20\penalty0 (6):\penalty0
  2807--2832, 2010.

\bibitem[Shamir and Zhang(2013)]{shamir2013stochastic}
Ohad Shamir and Tong Zhang.
\newblock Stochastic gradient descent for non-smooth optimization: Convergence
  results and optimal averaging schemes.
\newblock In \emph{Proceedings of The 30th International Conference on Machine
  Learning}, pages 71--79, 2013.

\bibitem[Steinhardt and Liang(2014)]{steinhardt2014adaptivity}
Jacob Steinhardt and Percy Liang.
\newblock Adaptivity and optimism: An improved exponentiated gradient
  algorithm.
\newblock In \emph{Proceedings of the International Conference on Machine
  Learning}, 2014.

\bibitem[Tibshirani(1996)]{tibshirani1996regression}
Robert Tibshirani.
\newblock Regression shrinkage and selection via the lasso.
\newblock \emph{Journal of the Royal Statistical Society. Series B
  (Methodological)}, pages 267--288, 1996.

\bibitem[Tibshirani et~al.(2012)Tibshirani, Bien, Friedman, Hastie, Simon,
  Taylor, and Tibshirani]{tibshirani2012strong}
Robert Tibshirani, Jacob Bien, Jerome Friedman, Trevor Hastie, Noah Simon,
  Jonathan Taylor, and Ryan~J Tibshirani.
\newblock Strong rules for discarding predictors in lasso-type problems.
\newblock \emph{Journal of the Royal Statistical Society: Series B (Statistical
  Methodology)}, 74\penalty0 (2):\penalty0 245--266, 2012.

\bibitem[Toulis et~al.(2014)Toulis, Rennie, and Airoldi]{toulis2014statistical}
Panos Toulis, Jason Rennie, and Edoardo Airoldi.
\newblock Statistical analysis of stochastic gradient methods for generalized
  linear models.
\newblock In \emph{ICML}, 2014.

\bibitem[Van~de Geer(2008)]{van2008high}
Sara~A Van~de Geer.
\newblock High-dimensional generalized linear models and the lasso.
\newblock \emph{The Annals of Statistics}, pages 614--645, 2008.

\bibitem[Van De~Geer and B{\"u}hlmann(2009)]{van2009conditions}
Sara~A Van De~Geer and Peter B{\"u}hlmann.
\newblock On the conditions used to prove oracle results for the lasso.
\newblock \emph{Electronic Journal of Statistics}, 3:\penalty0 1360--1392,
  2009.

\bibitem[Xiao(2010)]{xiao2010dual}
Lin Xiao.
\newblock Dual averaging methods for regularized stochastic learning and online
  optimization.
\newblock \emph{Journal of Machine Learning Research}, 11\penalty0
  (2543-2596):\penalty0 4, 2010.

\bibitem[Xu et~al.(2009)Xu, Huang, Fox, Patterson, and Jordan]{xu2009detecting}
Wei Xu, Ling Huang, Armando Fox, David Patterson, and Michael~I Jordan.
\newblock Detecting large-scale system problems by mining console logs.
\newblock In \emph{Proceedings of the ACM SIGOPS 22nd symposium on Operating
  systems principles}, pages 117--132. ACM, 2009.

\bibitem[Yang et~al.(2010)Yang, Xu, King, and Lyu]{yang2010online}
Haiqin Yang, Zenglin Xu, Irwin King, and Michael~R Lyu.
\newblock Online learning for group lasso.
\newblock In \emph{Proceedings of the 27th International Conference on Machine
  Learning (ICML-10)}, pages 1191--1198, 2010.

\bibitem[Zhao and Yu(2006)]{zhao2006model}
Peng Zhao and Bin Yu.
\newblock On model selection consistency of lasso.
\newblock \emph{The Journal of Machine Learning Research}, 7:\penalty0
  2541--2563, 2006.

\end{thebibliography}

\newpage

\begin{appendix}

\section{Proofs}

\subsection{Proofs for Section \ref{sec:main}}

\begin{proof}[Proof of Corollary \ref{cor:adaptive}]
We can check that the weights obtained by using the regularizer from \eqref{eq:psi1} can
equivalently be obtained using\footnote{It may seem surprising to let the regularizer
$\psi_t$ depend on $w_t$ as in \eqref{eq:psi-def}. However, we emphasize that Proposition
\ref{prop:adaptive} is a generic fact about convex functions, and holds for any
(random or deterministic) sequence of inputs.}
\begin{equation}
\label{eq:psi-def}
\psi_t(w) = \frac{\epsilon}{2}\|w\|_2^2 + \frac{\eta}{2}\sum_{s=1}^{t} \|w-w_s\|_2^2 + \lambda_t\|w\|_1.
\end{equation}
We also note that
$D_{\psi_t^*}(\theta_{t+1} || \theta_t) \leq \frac{1}{2(\epsilon + \eta t)}\|\theta_{t+1}-\theta_t\|_2^2$, 
which holds because 
$\psi_t$ is $(\epsilon + \eta t)$-strongly convex (see Lemma 2.19 of \citet{shalev2011}).
The inequality \eqref{eq:meta} then follows directly by applying 
Proposition \ref{prop:adaptive} to \eqref{eq:psi-def}.
\end{proof}

\subsection{Proofs for Section \ref{sec:sparse}}

\begin{proof}[Proof of Lemma~\ref{lem:bregman-restricted}]
We begin by noting that, given our regularizers $\psi_t$,
$w_{t, \, i} = 0$ if and only if $\Abs{\theta_{t, \, i}} \leq \lambda_t$.
Now, define
\[ \psi^+_t(w) = \left\{ \begin{array}{ccl} \psi_t(w) & : & w_i = 0 \textrm{ for all } i \not\in S_t \\ \infty & : & \textrm{else.} \end{array} \right.  \]
By construction, running adaptive mirror descent with the regularization sequence 
$\psi^+_t$ yields an identical set of iterates $\theta_t$ as running with the sequence $\psi_t$.
Moreover, because we also know that all non-zero coordinates of $w_t$ are 
contained in $S_{t-1}$, we can verify that
$$\psi^+_{t - 1}\p{w_t} - \psi^+_t\p{w_t} = \psi_{t - 1}\p{w_t} - \psi_t\p{w_t}, $$
and so using the $\psi^+_t$ leaves the regret bound \eqref{eq:adaptive} from Proposition 
\ref{prop:adaptive} unchanged except for the Bregman divergence terms 
$\sum_{t=1}^T D_{(\psi^+_t)^*}(\theta_{t+1} || \theta_t)$. We can thus bound the 
regret in terms of $D_{(\psi^+_t)^*}(\theta_{t+1} || \theta_t)$ rather than $D_{\psi_t^*}(\theta_{t+1} || \theta_t)$. 
On the other hand, we see that
\begin{align*}
(\psi^+_t)^*(\theta)
&= \sup_w \left\{ \left\langle w, \, \theta\right\rangle - \psi^+_t(w) \right\} \\
&= \sup_{w\sm{[S_t]}} \left\{ \left\langle w\sm{[S_t]}, \, \theta\sm{[S_t]}\right\rangle - \psi_t(w\sm{[S_t]}) \right\} \\
&= \psi_t^*(\theta\sm{[S_t]}),
\end{align*}
where $w\sm{[S]}$ and $\theta\sm{[S]}$ denote vectors that are zero on all coordinates not in $S$.\footnote{The last 
inequality makes use of the condition that $\psi_t(w\sm{[S_t]}, \tilde{w}\sm{[\neg S_t]})$ is minimized at $\tilde{w}\sm{[\neg S_t]} = 0$.}

The upshot is that 
\begin{align*}
D_{(\psi^+_t)^*}&(\theta_{t+1} || \theta_t) \\
&= \psi_t^*(\theta_{t+1}\sm{[S_t]}) - \left\langle \partial \psi_t^*(\theta_t\sm{[S_t]}), \theta_{t+1}\sm{[S_t]} - \theta_t\sm{[S_t]} \right\rangle - \psi_t^*(\theta_t\sm{[S_t]}) \\
&= D_{\psi_t^*}(\theta_{t+1}\sm{[S_t]} || \theta_t\sm{[S_t]}),
\end{align*}
as was to be shown.
\end{proof}
\begin{proof}[Proof of Lemma \ref{lemm:sparse_amd}]
We directly invoke Lemma~\ref{lem:bregman-restricted}. 
First, we check that its conditions are satisfied for 
$S_t = \cup_{s=1}^{t+1} \supp(w_s)$. Clearly the first two conditions are 
satisfied by construction, and for the third condition, we note that each 
term in $\psi_t(w)$ is either of the form $\|w\|_1$, which pushes all 
coordinates closer to zero, or $\|w_s-w\|_2^2$ with $s \leq t$, which pushes 
all coordinates outside of $\supp(w_s)$ closer to zero. Therefore, the third 
condition is also satisfied.

Now, we apply the result of Lemma~\ref{lem:bregman-restricted}.
The $\psi_T(u)$ 
term in \eqref{eq:bregman-restricted} yields 
\[ \frac{\epsilon}{2}\|u\|_2^2 + \frac{\eta}{2}\sum_{t=1}^T \|w_t-u\|_2^2 + \lambda_T \|u\|_1, \]
while the $\sum_{t=1}^T [\psi_{t-1}(w_t) - \psi_t(w_t)]$ term yields 
\[ \sum_{t=1}^T (\lambda_{t-1}-\lambda_t)\|w_t\|_1. \]
The most interesting term is 
the summation $\sum_{t=1}^T D_{\psi_t^*}(\theta_{t+1}\sm{[S_t]} \| \theta_t\sm{[S_t]})$.
By standard results on Bregman divergences, we know that if $\psi_t$ is 
$\gamma$-strongly convex, then $\psi_t^*$ is \emph{$\frac{1}{\gamma}$-strongly 
smooth} in the sense that
\[ D_{\psi_t^*}(x \| y) \leq \frac{1}{2\gamma}\|x-y\|_2^2. \]
In our case, $\psi_t$ is $(\epsilon + \eta t)$-strongly convex, 
so 
\begin{align*}
D_{\psi_t^*}(\theta_{t+1}\sm{[S_t]} \| \theta_t\sm{[S_t]}) &\leq \frac{1}{2\gamma}\|\theta_{t+1}\sm{[S_t]} - \theta_t\sm{[S_t]}\|_2^2 \\
 &\leq \frac{|S_t|}{2(\epsilon + \eta t)}\|\theta_{t+1} - \theta_t\|_{\infty}^2 \\
 &\leq \frac{k_t}{2(\epsilon + \eta t)}B^2,
\end{align*}
from which the lemma follows.
\end{proof}

\begin{proof}[Proof of Theorem \ref{theo:strong_sparse}]
By invoking Lemma~\ref{lemm:sparse_amd}, we have
\begin{align}
\label{eq:strong_sparse_proof1}
\sum_{t=1}^T \partial f_t \p{w_t}^{\top} \p{w_t - u} \leq \Omega + \Lambda + \frac{\alpha}{2} \sum_{t=1}^T \Norm{u - w_t}_2^2.
\end{align}

But now, because $f_t(\cdot)$ is $\alpha$-strongly convex, we also know that
$$ f_t\p{u} \geq f_t\p{w_t} +  \partial f_t\p{w_t}^\top \p{u - w_t} + \frac{\alpha}{2} \Norm{u - w_t}_2^2, $$
implying that
$$ \sum_{t = 1}^T \p{f_t\p{w_t} - f_t\p{u}} \leq \sum_{t = 1}^T \p{\partial f_t \p{w_t}^\top \p{u - w_t} - \frac{\alpha}{2} \Norm{u - w_t}_2^2}. $$
Chaining this inequality with \eqref{eq:strong_sparse_proof1}
gives us \eqref{eq:strong_sparse}.
\end{proof}

\subsection{Proofs for Section \ref{sec:statistical}}

Throughout our argument, we will bound certain quantities in terms of themselves. The following 
auxiliary lemma will be very useful in turning these implicit bounds into explicit bounds.

\begin{lemma}
\label{lemm:self-bound}
Suppose that $a,b,c \geq 0$ and $S \leq a + \sqrt{bS+c^2}$.
Then $S \leq 2a + b + c$.
\proof
We have $(S-a)^2 \leq bS+c^2$, so that $S^2 - (2a+b)S + a^2-c^2 \leq 0$. This implies that
\begin{align*}
S &\leq \frac{(2a+b) + \sqrt{(2a+b)^2 - 4(a^2-c^2)}}{2} \\
  &= (2a+b)\frac{1 + \sqrt{1 - 4(a^2-c^2)/(2a+b)^2}}{2} \\
 &\leq (2a+b)\frac{1 + \sqrt{1+4c^2/(2a+b)^2}}{2} \\
 &\leq (2a+b)\frac{2 + \sqrt{4c^2}/(2a+b)}{2} \\
 &= 2a+b+c,
\end{align*}
as claimed. The final inequality uses the fact that $\sqrt{x+y} \leq \sqrt{x} + \sqrt{y}$.
\endproof
\end{lemma}

It will also be useful to have the following adaptive 
variant of Azuma's inequality. 
Throughout, we use $\lg(x)$ to denote the base-$2$ logarithm of $x$.
In interpreting the lemma below, it will be helpful to think of 
$Z$ as a sum of $T$ independent zero-mean random variables $X_{1:T}$, so that 
$\EE{Z \cond \ff_t} - \EE{Z} = X_1+\cdots+X_t$, and to think of $M_t$ as a 
bound on $|X_t|$ that is allowed to depend on $X_{1:t-1}$.
\begin{lemma}
\label{lemm:better_azuma}
Let $Z$ be a $\ff_T$-measurable random variable, and let
$$\{\emptyset\} = \ff_0 \subseteq \ff_1 \subseteq ... \subseteq \ff_T$$
be a filtration such that
$$ \EE{Z \cond \ff_{t}} - \EE{Z \cond \ff_{t - 1}} \in [A_t, B_t] 
    \text{ almost surely for } t =  1, \, ..., \, T, $$
where $(A_t,B_t)$ is $\ff_{t - 1}$-measurable, and let 
$M_t = \tfrac{1}{2}(B_t-A_t)$. Moreover, suppose that 
$\sup_{t=1}^T M_t \leq \sigma_1 \sigma_2$ with probability $1$ and $\sigma_1 \geq 1$. Then, for all $\delta > 0$,
with probability $1-\delta$, we have
$$ \EE{Z \cond \ff_t} - \EE{Z} \geq -\sqrt{\log\p{\frac{\lg(2\sigma_1^2 T)}{\delta}}\max\p{2\sigma_2^2, \frac{9}{4}\sum_{s=1}^t M_s^2}},$$
for all $t \leq T$.
\proof
Let $\Delta_t = \EE{Z \cond \ff_t} - \EE{Z \cond \ff_{t - 1}}$. Note that we 
have
\begin{align*}
\EE{\exp\p{-c \Delta_t - \frac{c^2M_t^2}{2}} \cond \ff_{t-1}} &= \EE{\exp\p{-c \Delta_t} \cond \ff_{t-1}}\exp\p{-\frac{c^2M_t^2}{2}} \\
 &\leq \exp\p{\frac{c^2M_t^2}{2}}\exp\p{-\frac{c^2M_t^2}{2}} \\
 &= 1
\end{align*}
for all $c > 0$. Therefore, 
$$ Y_t^c \eqdef \exp\p{-c\sum_{s = 1}^t \Delta_s - \frac{c^2}{2} \sum_{s = 1}^t M_s^2} $$
is a supermartingale, and so $\PP{\sup_{t=1}^T Y_t^c \geq \exp\p{\frac{\gamma c^2}{2}}} \leq \exp\p{-\frac{\gamma c^2}{2}}$. Noting that $\sum_{s=1}^t \Delta_s = \EE{Z \cond \ff_t} - \EE{Z}$, we 
then have that the probability that $\EE{Z \cond \ff_t} - \EE{Z} < -\frac{c}{2}\p{\gamma + \sum_{s=1}^t M_s^2}$ for any $t$ is at most $\exp\p{-\frac{\gamma c^2}{2}}$.

To finish the proof, we will optimize over $\gamma$ and $c$. The problem is that 
the optimal values of $\gamma$ and $c$ depend on the $M_t$, so we need some way 
to identify a small number of $(\gamma, c)$ pairs over which to union bound.

To start, we want
$\exp\p{-\frac{\gamma c^2}{2}}$ to be at most $\delta$, so for a fixed 
$\gamma > 0$ we will set $c = \sqrt{2\log(1/\delta)/\gamma}$, leading to the 
bound
\begin{align}
\label{eq:gamma-1}
&\mathbb{P}\bigg[\EE{Z \cond \ff_t} - \EE{Z} \\
\notag
&\ \ \ \ \leq -\sqrt{2\log(1/\delta)} \frac{\sqrt{\gamma} + \sqrt{1/\gamma}\sum_{t=1}^T M_t^2}{2} \text{ for any $t$}\bigg] \leq \delta.
\end{align}
For $\gamma \in \left[\sum_{t=1}^T M_t^2, 2\sum_{t=1}^T M_t^2\right]$, we have
$\sqrt{\gamma} + \sqrt{1/\gamma}\sum_{s=1}^t M_s^2 \leq (\sqrt{2}+\sqrt{1/2})\sqrt{\sum_{s=1}^t M_s^2}$, which yields 
\begin{equation}
\label{eq:gamma-2}
\PP{\EE{Z \cond \ff_t} - \EE{Z} \leq -\frac{3}{2}\sqrt{\log(1/\delta)\sum_{t=1}^T M_t^2} \text{ for any $t$}} \leq \delta.
\end{equation}
Now, we know that $\sum_{t=1}^T M_t^2 \leq \sigma_1^2\sigma_2^2 T$, so we will union bound over 
$\gamma \in \{\sigma_2^2, 2\sigma_2^2, \ldots, \sigma_2^22^{\lceil \lg(\sigma_1^2 T/2) \rceil}\}$, 
which is $\max\p{1, \lceil \lg(\sigma_1^2 T) \rceil} \leq \lg(2\sigma_1^2 T)$ 
values of $\gamma$ in total. From this, we have the desired bound as long as 
$\sum_{s=1}^t M_s^2 \geq \sigma_2^2$. To finish, note that, if 
$\sum_{s=1}^t M_s^2 < \sigma_2^2$, then for $\gamma = \sigma_2^2$ we have, by 
\eqref{eq:gamma-1}, 
\begin{align}
\label{eq:gamma-3}
\lefteqn{\PP{\EE{Z \cond \ff_t} - \EE{Z} \leq -\sqrt{2\log(1/\delta)\sigma_2^2}}} \\
\notag
  &\phantom{++}\leq \PP{\EE{Z \cond \ff_t} - \EE{Z} \leq -\sqrt{2\log(1/\delta)}\tfrac{\sigma_2 + (1/\sigma_2)\cdot \sum_{t=1}^T M_t^2}{2}} \\
\notag
  &\phantom{++}\leq \delta.
\end{align}
Combining \eqref{eq:gamma-2} and \eqref{eq:gamma-3} and decreasing $\delta$ 
by a factor of $\lg(2\sigma_1^2 T)$ for the union bound completes the proof.
\endproof
\end{lemma}

\begin{lemma}
\label{lemm:variance-bound}
For $t=1,\ldots,T$, let $z_t \in \partial f_t(w_t\sm{[S]})$ and $z_t' \in \partial f_t(w^*)$.
Then, using notation from Lemma \ref{lemm:l1_cost}, with probability $1-\delta$ we have
\begin{align}
\label{eq:variance_bound}
\lefteqn{\sum_{s=1}^t \p{f_s\p{w_s\sm{[S]}} - f_s\p{w^*}} \leq \p{\sum_{s=1}^t \Lstar\p{w_s\sm{[S]}}-\Lstar\p{w^*}}} \\
\notag
 & \phantom{+++} - \sqrt{kB^2\log\p{\frac{\lg(2T)}{\delta}}\max\p{\frac{2kB^2}{\alpha^2},\frac{9}{4}\sum_{s=1}^t \Norm{w_s\sm{[S]}-w^*}_2^2}}
\end{align}
for all $t \leq T$.
\end{lemma}
\begin{proof}
Suppose $z_t \in \partial f_t(w_t\sm{[S]})$. Then,
\begin{align*}
f_t\p{w_t\sm{[S]}}-f_t\p{w^*} &\leq z_t^{\top}\p{w_t\sm{[S]}-w^*} \\
 &\leq \Norm{z_t\sm{[S]}}_2 \Norm{w_t\sm{[S]}-w^*}_2 \\
 &\leq \sqrt{k}B \Norm{w_t\sm{[S]}-w^*}_2.
\end{align*}
Similarly, by considering $z_t' \in \partial f_t(w^*)$, we find that
\begin{align*}
f_t\p{w_t\sm{[S]}} - f_t\p{w^*}
 &\geq -\sqrt{k}B \Norm{w_t\sm{[S]}-w^*}_2.
\end{align*}
Note also that $\|w_t\sm{[S]}-w^*\|_2 \leq \sqrt{k}B/\alpha$. 
Now, let $Z = \sum_{t=1}^T f_t(w_t\sm{[S]}) - f_t(w^*)$ and invoke 
Lemma~\ref{lemm:better_azuma}. We then have 
$$\Delta_t \in -\p{\law\p{w_t\sm{[S]}}-\law\p{w^*}} + \left[-\sqrt{k}B\Norm{w_t\sm{[S]}-w^*}_2, \sqrt{k}B\Norm{w_t\sm{[S]}-w^*}_2\right], $$
hence 
$M_t = \sqrt{k}B\Norm{w_t\sm{[S]}-w^*}_2$, and we can set $\sigma_2 = kB^2/\alpha$, $\sigma_1 = 1$, from 
which the result follows.
\end{proof}

\begin{proof}[Proof of Lemma \ref{lemm:l1_cost}]

We begin by noting that
\begin{align*}
\Lambda
&= \sum_{t = 1}^T \p{\lambda_{t} - \lambda_{t - 1}}  \p{ \Norm{w^*}_1 - \Norm{w_t}_1} \\
&\leq \sum_{t = 1}^T \p{\lambda_{t} - \lambda_{t - 1}}  \p{ \Norm{w^*}_1 - \Norm{w_t\sm{[S]}}_1}.
\end{align*}
With our regularization schedule $\lambda_t = \lambda \sqrt{t + 1}$, we can check that $\lambda_t - \lambda_{t - 1} \leq \lambda / (2\sqrt{t})$. Thus, by Cauchy-Schwarz,
\begin{align}
\notag
\Lambda
&\leq \sqrt{\p{\sum_{t = 1}^T \p{\lambda_t - \lambda_{t - 1}}^2} \, \p{\sum_{t = 1}^T \p{\Norm{w^*}_1 - \Norm{w_t\sm{[S]}}}^2}} \\
\notag
&\leq \frac{\lambda}{2} \, \sqrt{\p{1 + \log T} \sum_{t = 1}^T \Norm{w_t\sm{[S]} - w^*}_1^2} \\
\notag
&\leq \frac{\lambda}{2} \, \sqrt{k\p{1 + \log T} \sum_{t = 1}^T \Norm{w_t\sm{[S]} - w^*}_2^2}.
\end{align}

Now, using the strong convexity of $\Lstar\p{\cdot}$ on $S$ as well as 
Lemma~\ref{lemm:variance-bound},
we can verify that, with probability $1-\delta$,
\begin{align*}
&\sum_{t=1}^T \Norm{w_t\sm{\sm{[S]}}-w^*}_2^2 \\
\notag
&\leq \frac{2}{\alpha}\sum_{t=1}^T \p{\Lstar\p{w_t\sm{[S]}} - \Lstar\p{w^*}} \\
 &\leq \frac{2}{\alpha}\Bigg(\sum_{t=1}^T \p{f_t\p{w_t\sm{[S]}} - f_t\p{w^*}} \\
 & \ \ \ \ + \sqrt{kB^2\log\p{\frac{\lg(2T)}{\delta}}\max\p{\frac{2kB^2}{\alpha^2}, \frac{9}{4}\sum_{t=1}^T \Norm{w_t\sm{[S]}-w^*}_2^2}}\Bigg).
\end{align*}
By Lemma~\ref{lemm:self-bound} we thus have
\begin{align}
\label{eq:cauchy_bound3}
\lefteqn{\sum_{t=1}^T \|w_t\sm{[S]}-w^*\|_2^2} \\
\notag
 &\leq \frac{4}{\alpha} \sum_{t=1}^T \p{f_t\p{w_t\sm{[S]}} - f_t\p{w^*}}  \\
\notag 
 &\ \ \ \ \ \ + \max\left\{
\frac{\sqrt{2\log(\lg(2T)/\delta)}kB^2}{\alpha^{2}},
\frac{9kB^2\log(\lg(2T)/\delta)}{4\alpha^2}\right\} \\
\notag
 &= \frac{4}{\alpha} \sum_{t=1}^T \p{f_t\p{w_t\sm{[S]}} - f_t\p{w^*}} 
   + \frac{9kB^2\log(\lg(2T)/\delta)}{4\alpha^2}.
\end{align}
Plugging this inequality into our previous bound for $\Lambda$ yields the desired result.
\end{proof}

\begin{proof}[Proof of Theorem \ref{theo:l1_recipe}]
We start by applying our bound on $\Lambda$ from Lemma \ref{lemm:l1_cost} 
to \eqref{eq:l1_recipe_init}.
We have that, with probability $1 - \delta$,
\begin{align*}
&\sum_{t = 1}^T \p{f_t\p{w_t} - f_t\p{w^*}} \leq R_T\p{w^*} + \\
& \ \ \frac{\lambda}{2} \, \sqrt{k \p{1 + \log T} \p{\frac{4}{\alpha} \sum_{t = 1}^T \p{f_t\p{w_t\sm{[S]}} - f_t \p{w^*}} +\frac{9 k B^2 \log \p{\log_2(2T)/\delta}}{4\alpha^2}}}.
\end{align*}
The excess loss we incur from using non-zero weights outside the set $S$ is
$$ \Delta = \sum_{t = 1}^T \p{f_t\p{w_t} - f_t\p{w_t\sm{[S]}}}. $$
We split our analysis into two cases depending on the sign of $\Delta$. 
Also let $r = \sum_{t=1}^T \p{f_t\p{w_t} - f_t\p{w^*}}$ denote the 
quantity we want to bound. 

When $\Delta \geq 0$, we can use the fact that the sum inside the square root is 
equal to $r - \Delta$, and loosen the inequality to
$$r \leq R_T\p{w^*} + \frac{\lambda}{2} \, \sqrt{k \p{1 + \log T} \p{\frac{4}{\alpha} r +\frac{9 k B^2 \log \p{\log_2(2T)/\delta}}{4\alpha^2}}}. $$
Since $r$ appears on both sides of the inequality, we can use 
Lemma \ref{lemm:self-bound} to show that
$$ r \leq 2 R_T\p{w^*} + \frac{ k \lambda^2\p{1 + \log T}}{\alpha} + \frac{3Bk\lambda}{4\alpha} \sqrt{\p{1 + \log T}\log\p{\lg(2T) / \delta}}, $$
which yields the desired expression via the AM-GM inequality
$$ \frac{3B\lambda}{4}\sqrt{(1+\log T)\log(\lg(2T)/\delta)} \leq \frac{1}{3}\lambda^2(1+\log T) + \frac{1}{2}B^2\log(\lg(2T)/\delta). $$
Meanwhile, if $\Delta < 0$, we write
\begin{align*}
&\sum_{t = 1}^T \p{f_t\p{w_t\sm{[S]}} - f_t\p{w^*}} \leq - \Delta + R_T\p{w^*}  + \\
& \ \ \frac{\lambda}{2} \, \sqrt{k \p{1 + \log T} \p{\frac{4}{\alpha} \sum_{t = 1}^T \p{f_t\p{w_t\sm{[S]}} - f_t \p{w^*}} +\frac{9 k B^2 \log \p{\lg(2T)/\delta}}{4\alpha^2}}}.
\end{align*}
Again, applying Lemma \ref{lemm:self-bound}, we get
\begin{align*}
&\sum_{t = 1}^T \p{f_t\p{w_t\sm{[S]}} - f_t\p{w^*}} \leq - 2 \Delta + 2 R_T\p{w^*} \\
&\ \ \ \ + \frac{ k \lambda^2\p{1 + \log T}}{\alpha} + \frac{3Bk\lambda}{4\alpha} \sqrt{\p{1 + \log T }\log\p{\lg(2T) / \delta}}.
\end{align*}
If we put one of the two $\Delta$ factors back on the left-hand side of the inequality, we 
get the desired expression via the same AM-GM inequality as before.
\end{proof}

\begin{proof}[Proof of Lemma \ref{lemm:pure_noise}]
We again apply Lemma~\ref{lemm:better_azuma}. In this case, for any $j \not\in S$, we let $z_t^{j}$ denote the $j$th coordinate of $z_t$. Then, set 
$Z = \sum_{t=1}^T z_t^{j}$ and let $\ff_t$ be the sigma-algebra generated by $f_{1:t}$. 
Clearly we can take $A_t = -B$, $B_t = B$, and $M_t = B$, and set $\sigma_2 = B$, $\sigma_1 = 1$. 
Then, by applying Lemma~\ref{lemm:better_azuma} in both directions, we get
$$ \PP{\left|\sum_{s=1}^t z_s^j\right| \geq \sqrt{\log\p{\frac{2d\lg(2T)}{\delta}}\max\p{2B^2, \frac{9}{4}B^2t}} \text{ for any $t$}} \leq \delta / d. $$
Simplifying $\max\p{2B^2, \frac{9}{4}B^2t}$ to $\frac{9}{4}B^2t$ and 
applying the union bound over all coordinates $j\not\in S$ then yields the desired result.
\end{proof}

\subsection{Proofs for Section \ref{sec:expected_cvx}}
\label{sec:apx_expected_cvx}

\begin{proof}[Proof of Lemma \ref{lemm:expected_cvx}]
Define $D_t = f_t(w_t) - f_t(u) - \partial f_t(w_t)^{\top}\p{w_t-u} + \frac{\alpha}{4}\Norm{w_t-u}^2$, and let 
$X_t = \sum_{s=1}^t D_s$. The main idea is to show that $X_t$ is a random walk 
with negative drift, from which we can then use standard martingale cumulant 
techniques to bound $\sup_{t=1}^{T} X_t$, which is what we need to do 
in order to establish \eqref{eq:expected_cvx}.

First note that, by the Lipschitz assumption on $f_t$, we have 
$$ \Abs{f_t\p{w_t} - f_t\p{u}} \leq L \Norm{w_t-u} \; \eqand \;
\Abs{\partial f_t\p{w_t}^{\top}\p{w_t-u}} \leq L \Norm{w_t-u}, $$
hence $D_t \in \left[\frac{\alpha}{4}\Norm{w_t-u}^2-2L\Norm{w_t-u}, \, \frac{\alpha}{4}\Norm{w_t-u}^2+2L\Norm{w_t-u}\right]$. Furthermore, we have 
\begin{align*}
&\EE{D_t \mid \sF_{t-1}} \\
&\ \ = \EE{f_t\p{w_t} - f_t\p{u} - \partial f_t\p{w_t}^{\top}\p{w_t-u} + \frac{\alpha}{4}\Norm{w_t-u}^2 \mid \sF_{t-1}}\\
 &\ \ \leq \EE{-\frac{\alpha}{2} \Norm{w_t-u}^2 + \frac{\alpha}{4}\Norm{w_t-u}^2 \ \middle|\ \sF_{t-1}} \; \; \;   \text{(by strong convexity)} \\
 &\ \ = -\frac{\alpha}{4} \Norm{w_t-u}^2.
\end{align*}
We next put these together and start going through the standard Chernoff argument: for any $0 \leq \lambda \leq \tfrac{\alpha}{8L^2}$,
\begin{align*}
\lefteqn{\EE{\exp\p{\lambda X_t} \mid \sF_{t-1}}} \\
 &= \EE{\exp\p{\lambda X_{t-1}}\exp\p{\lambda D_t} \mid \sF_{t-1}} \\
 &\leq \exp\p{\lambda X_{t-1}} \exp\p{\lambda \EE{D_t} + 2\lambda^2 L^2 \Norm{w_t-u}^2} \\
 &\leq \exp\p{\lambda X_{t-1}}\exp\p{-\lambda\frac{\alpha}{4}\Norm{w_t-u}^2 + 2\lambda^2L^2\Norm{w_t-u}^2} \\
 &\leq \exp\p{\lambda X_{t-1}},
\end{align*}
where the second inequality follows from the sub-Gaussianity of bounded random variables.
Hence, for $\lambda = \frac{\alpha}{8L^2}$, $\exp(\lambda X_t)$ 
is a non-negative supermartingale with $\exp(\lambda X_0) = 1$. By the optional 
stopping theorem and Markov's inequality, we then have
$$\PP{\sup_{t=1}^{T} X_t \geq M} \leq \exp\p{-\lambda M}$$
and so, with probability $1-\delta$, $X_t$ never goes above 
$$\frac{\log(1/\delta)}{\lambda} = \frac{8L^2\log(1/\delta)}{\alpha},$$
as was to be shown.
\end{proof}


\begin{proof}[Proof of Theorem \ref{theo:sparse_expected_cvx}]
Recall that we are running adaptive mirror descent using the regularizers 
from \eqref{eq:psi-def}, 
which corresponds to setting 
\[ \psi_t(w) = \frac{\epsilon}{2}\Norm{w}_2^2 + \frac{\alpha}{4} \sum_{s=1}^{t} \Norm{w-w_s}_2^2 + \lambda_t \|w\|_1. \]
Note that $\psi_t$ is $(\epsilon + (\alpha/2)t)$-strongly convex with respect 
to the $L_2$ norm. Also note that, since $\|\partial f_t\|_{\infty} \leq B$ and 
$\supp(w_t) \subseteq S$, $f_t$ is $(B\sqrt{k})$-Lipschitz, at least over the 
space $\sH$ of $w_t$ with $\supp(w_t) \subseteq S$, $\|w_t\|_1 \leq R$.

Plugging into the regret bound from 
Lemma~\ref{lemm:sparse_amd} and applying Lemma~\ref{lemm:expected_cvx}, we get 
\begin{align}
\sum_{t=1}^T &\p{f_t(w_t) - f_t(u) + \frac{\alpha}{4}\Norm{w_t-u}_2^2} \\
\notag
&\leq \frac{\epsilon}{2}\Norm{u}_2^2 + \frac{\alpha}{4} \sum_{t=1}^T \Norm{w_t-u}_2^2 + \frac{B^2k}{2}\sum_{t=1}^T \frac{1}{\epsilon + (\alpha/2)t} \\
\notag
&\ \ + \frac{8B^2k\log(1/\delta)}{\alpha} + \Lambda.
\end{align}
Subtracting $\sum_{t=1}^T \frac{\alpha}{4}\Norm{w_t-u}_2^2$ from both sides 
and using the fact that $\epsilon = 0$, 
$\sum_{t=1}^T \frac{1}{t} \leq 1 + \log(T)$ yields the desired bound.
\end{proof}

\begin{proof}[Proof of Theorem~\ref{thm:main-risk}]
We will prove the following slightly more precise result. Under the stated conditions,
with probability $1-\delta$, we 
have $\supp(w_t) \subseteq S$ for all $t$, and 
\begin{align}
\label{eq:main-risk-2}
\sum_{t=1}^T \p{f_t(w_t) - f_t(w^*)} 
&\leq \frac{22kB^2(1+\log T)}{\alpha}\log\p{\frac{6d\lg(2T)}{\delta}}.
\end{align}
To establish this result, we will union bound over three events, each of which holds with 
probability $1-\delta/3$. 
First, by Lemma~\ref{lemm:pure_noise}, we know that 
$\supp(w_t) \subseteq S$ for all $t$ with probability 
$1-\delta/3$. Therefore, by Theorem~\ref{theo:sparse_expected_cvx} we have, 
with overall probability $1-2\delta/3$,
\begin{align*}
\sum_{t=1}^T \p{f_t(w_t) - f_t(w^*)}
&\leq \frac{kB^2}{\alpha}\p{1+\log T + 8\log\p{3/\delta}} + \Lambda.
\end{align*}
Finally, invoking Theorem~\ref{theo:l1_recipe}, we have, 
with overall probability $1-\delta$,
\begin{align*}
\sum_{t=1}^T &\p{f_t(w_t) - f_t(w^*)} \\
\notag
&\leq \frac{2kB^2}{\alpha}\p{1 + \log T + 8\log\p{\frac{3}{\delta}}} \\
\notag
&\ \ \ +\frac{kB^2}{\alpha}\Bigg(3\log\p{\frac{6d\lg(2T)}{\delta}}(1+\log T) 
+ \frac{1}{2}\log\p{\frac{3\lg(2T)}{\delta}}\Bigg) \\
\notag
&\leq \frac{22kB^2}{\alpha}(1+\log T)\log\p{\frac{6d\lg(2T)}{\delta}},
\end{align*}
which proves the theorem.
\end{proof}

\subsection{Proofs for Section \ref{sec:batch}}

We begin this section by stating a series of technical results that will lead us
to Theorem \ref{theo:main_batch}. We defer proofs of these results to
Section \ref{sec:batch_more_proofs}.
To warm up, we give the following analogue to Theorem \ref{theo:strong_sparse} without proof.

\begin{theorem}
\label{theo:strong_sparse_batch}
Suppose that we are given a sequence of $\alpha$-strongly convex losses, 
and that we run 
adaptive mirror descent on the losses $\tf_t$ with the regularizers $\tpsi_t$ from \eqref{eq:batch_reg}. Then, 
using notation from Lemma \ref{lemm:sparse_amd}, the weights $w_t$ learned by this algorithm satisfy
\begin{equation}
\label{eq:strong_sparse_batch}
\sum_{t = 1}^T t\p{f_t\p{w_t} - f_t\p{u}} \leq \frac{2B^2}{\alpha} \sum_{t = 1}^T k_t + \sum_{t = 1}^T \p{\lambda_{t - 1} - \lambda_t} (\Norm{w_t}_1 - \Norm{u}_1).
\end{equation}
\end{theorem}

We now proceed to extend the previous theorems from controlling $f_t$
to controlling $\tilde f_t = t f_t$.
Most of the results hold with only Assumptions (1-3); we only need Assumption 4 to ensure that 
$\supp(w_t) \subseteq S$ for all $t$. We state each result under the assumption that $\supp(w_t) \subseteq S$, 
and show at the end that this assumption holds with high probability under Assumption 4.

First, we need an excess risk bound that holds for functions that are strongly convex in expectation:
\begin{theorem}
\label{theo:sparse_expected_cvx_batch}
Suppose that the loss functions $f_t$ satisfy assumptions (1-3), and that we run adaptive mirror descent 
as in the statement of Theorem \ref{theo:strong_sparse_batch}. Suppose also that $\supp(w_t) \subseteq S$ 
for all $t$. Then, for any fixed $u$ and $\delta > 0$, with probability 
at least $1 - \delta$, the learned weights $w_t$ satisfy
\begin{align}
\label{eq:sparse_expected_cvx_batch}
\sum_{s = 1}^t &s\p{f_s\p{w_s} - f_s\p{u}} \leq \frac{2B^2kt}{\alpha} \\
\notag
&+ \frac{16B^2kt\log\p{\log_2(2T)/\delta}}{\alpha} 
+ \sum_{s = 1}^t \p{\lambda_{s - 1} - \lambda_s} (\Norm{w_s}_1 - \Norm{u}_1)
\end{align}
for all $t \leq T$.
\end{theorem}

We also need an analogue to Theorem \ref{theo:l1_recipe}, which bounds the cost of the $L_1$ terms.

\begin{theorem}
\label{theo:l1_recipe_batch}
Suppose that assumptions (1-3) hold and that
\begin{equation}
\label{eq:l1_recipe_init_batch}
\sum_{t = 1}^T t \p{f_t\p{w_t} - f_t\p{w^*}} \leq R_T\p{w^*} +  \lambda_T \Norm{w^*}_1 + \sum_{t = 1}^T \p{\lambda_{t - 1} - \lambda_t} \Norm{w_t}_1
\end{equation}
for some main term $R_T(w^*) \geq 0$. Suppose moreover that $\supp(w_t) \subseteq S$ for all $t$. 
Then, for regularization schedules of the form 
$\lambda_t = \lambda \cdot t^{3/2}$, the following excess 
risk bound also holds: 
\begin{align}
\label{eq:l1_recipe_goal_batch}
 \sum_{s = 1}^t s &\p{f_s\p{w_s} - f_s\p{w^*}} \\
 \notag
 &\leq 2 R\p{w^*}  +\frac{9 k \lambda^2 t}{\alpha}
+ \frac{3Bk\lambda t}{2 \alpha}\sqrt{9\log\p{\lg(2T^3)/\delta}} \\
\notag
 &\leq 2R\p{w^*} + \frac{kt}{\alpha}\p{10\lambda^2 + 6B^2\log\p{\frac{\lg(2T^3)}{\delta}}}.
\end{align}
\end{theorem}
Finally, we need a technical result analogous to Lemma~\ref{lemm:variance-bound} 
from before:
\begin{lemma}
\label{lemm:variance-bound-batch}
Suppose that the $f_t$ are $L$-Lipschitz over $\sH$ and $\alpha$-strongly 
convex (both with respect to the $L_2$-norm). Then, with probability $1-\delta$,
for all $t$ we have
\begin{align}
\lefteqn{\sum_{s=1}^t s\p{f_s\p{w_s} - f_s\p{w^*}} \geq \sum_{s=1}^t s\p{\law\p{w_s} - \law\p{w^*}}} \\
\notag
 &\phantom{+++} - \sqrt{L^2t\log\p{\frac{\lg(2T^3)}{\delta}}\max\p{\frac{2L^2}{\alpha^2},\frac{9}{4}\sum_{s=1}^t s\Norm{w_s-w^*}_2^2}}.
\end{align}
\end{lemma}
Each of the above results is proved later in this section, 
in \ref{sec:batch_more_proofs}.
These results give us the necessary scaffolding to prove 
Proposition~\ref{prop:batch} and Theorem~\ref{theo:main_batch}, 
which we do now.
\begin{proof}[Proof of Proposition \ref{prop:batch}]
The desired result follows by combining Lemma~\ref{lemm:variance-bound-batch} 
with the given excess loss bound. 
In particular, we first have, by Lemma~\ref{lemm:variance-bound-batch}, 
\begin{align*}
&\sum_{s=1}^t s\p{\law\p{w_s} - \law\p{w^*}} \\
 &\ \ \leq \sum_{s=1}^t s\p{f_s\p{w_s} - f_s\p{w^*}} \\
 &\ \ \phantom{+++} + \sqrt{L^2t\log\p{\frac{\lg(2T^3)}{\delta}}\max\p{\frac{2L^2}{\alpha^2},\frac{9}{4}\sum_{s=1}^t s\Norm{w_s-w^*}_2^2}} \\
 &\ \ \leq R_{\delta}t + \sqrt{L^2t\log\p{\frac{\lg(2T^3)}{\delta}}\max\p{\frac{2L^2}{\alpha^2}, \frac{9}{2\alpha}\sum_{s=1}^t s\p{\law\p{w_s}-\law\p{w^*}}}}.
\end{align*}
Using Lemma~\ref{lemm:self-bound}, we get
\begin{align*}
\sum_{s=1}^t &s\p{\law\p{w_s} - \law\p{w^*}} \\
&\leq 2R_{\delta}t + \max\p{\frac{\sqrt{2t\log(\lg(2T^3)/\delta)}L^2}{\alpha}, \frac{9t\log(\lg(2T)/\delta)L^2T}{2\alpha}} \\
 &\leq 2R_{\delta}T + \frac{9\log(\lg(2T)/\delta)L^2t}{2\alpha}.
\end{align*}
Finally, invoking the convexity of $\law$, we have
\begin{align*}
&\Lstar\p{\frac{2}{t(t+1)}\sum_{s=1}^t sw_s} - \Lstar(w^*) \\
&\ \ \leq \frac{2}{t(t+1)} \sum_{s=1}^t s\p{\Lstar(w_s)-\Lstar(w^*)} \\
 &\ \ \leq \frac{4}{t}R_{\delta} +  \frac{9\log(\lg(2T)/\delta)L^2}{\alpha t},
\end{align*}
as was to be shown.
\end{proof}

\begin{proof}[Proof of Theorem \ref{theo:main_batch}]
Using Lemma~\ref{lemm:better_azuma}, we can show that by using  
$$ \lambda = c_\delta = \frac{3B}{2}\sqrt{\log \p{\frac{2d\lg(2T^3)}{\delta}}}, $$
we have $\supp\p{w_t} \in S$ for all $t$. 

By Theorem \ref{theo:sparse_expected_cvx_batch} combined with 
Theorem \ref{theo:l1_recipe_batch}, we know that, for any $\delta > 0$, 
with probability $1 - \delta$,
\begin{align}
\label{eq:proof_step_one}
\sum_{s = 1}^t &s\p{f_s\p{w_s} - f_s\p{w^*}} \\
\notag
&\leq \frac{kt}{\alpha}\p{4B^2 + 32B^2\log\p{\log_2(2T)/\delta} +10 c_\delta^2 + 6B^2\log\p{\lg(2T^3) / \delta}} \\
\notag
 &\leq \frac{B^2kt}{\alpha}\log\p{\frac{2d\lg(2T^3)}{\delta}}\left(4 + 32 + 22.5 + 6\right) \\
\notag
 &\leq \frac{65B^2kt\log\p{2d\lg(2T^3)/\delta}}{\alpha}.
\end{align}
Thus, by Proposition~\ref{prop:batch}, with probability $1 - 2\delta$,
\begin{align*}
&\law\p{\frac{2}{t(t + 1)} \sum_{s = 1}^t s \, w_s} - \law\p{w^*} \\
&\ \ \ \ \leq  \frac{260B^2k\log\p{2d\lg(2T^3)/\delta}}{\alpha t} 
+ \frac{9B^2k\log\p{\lg(2T^3)/\delta}}{\alpha T} \\
&\ \ \ \ \leq \frac{269B^2k\log\p{2d\lg(2T^3)/\delta}}{\alpha t},
\end{align*}
which yields the desired result.
\end{proof}

\subsubsection{Technical Derivations}
\label{sec:batch_more_proofs}

\begin{proof}[Proof of Theorem \ref{theo:sparse_expected_cvx_batch}]
For the first part of the proof, we will show that, with probability 
$1-\delta$, 
\begin{align}
\label{eq:batch_exp_cvx_0}
&\sum_{s=1}^t t\p{f_t(w_t)-f_t(u) + \frac{\alpha}{4}\|w_t-u\|_2^2} \\
\notag
&\ \ \ \ \leq \sum_{s=1}^t t\,\partial f_t(w_t)^{\top}(w_t-u) + \frac{16B^2kt}{\alpha}\log\p{\frac{\log_2(2T)}{\delta}}
\end{align}
for all $t \leq T$. To begin, we note that $tf_t(w_t)$ is $tB\sqrt{k}$-Lipschitz 
and $\alpha t$ convex. Consequently, if we define $D_t$ to be 
$$ D_t \eqdef t\p{f_t(w_t)-f_t(u)-\partial f_t(w_t)^{\top}(w_t-u)+\frac{\alpha}{4}\|w_t-u\|_2^2},$$ 
we have
$D_t \in -\frac{t\alpha}{4}\|w_t-u\|_2^2 + \left[-2tB\sqrt{k}\|w_t-u\|_2,2tB\sqrt{k}\|w_t-u\|_2\right]$, and 
$\bE[D_t] \leq -\frac{t\alpha}{4}\|w_t-u\|_2^2$.

Now, arguing as before, if we let $X_t = \sum_{s=1}^t D_s$, then
\begin{align*}
\EE{\exp(\lambda X_t) \mid f_{1:t-1}} &= \exp(\lambda X_{t-1})\EE{\exp(\lambda D_t) \mid f_{1:t-1}} \\
 &\leq \exp(\lambda X_{t-1})\exp\p{-\lambda\frac{t\alpha}{4}\|w_t-u\|_2^2 + 2\lambda^2B^2k\|w_t-u\|_2^2} \\
 &\leq \exp(\lambda X_{t-1})
\end{align*}
provided that $0 \leq \lambda \leq \frac{\alpha}{8B^2kt}$. Hence, by the 
same martingale argument as before, we have that  
\begin{equation}
\label{eq:batch_exp_cvx_1}
\bP\left[\sup_{s=1}^t X_s \geq \frac{8B^2kt\log(\log_2(2T)/\delta)}{\alpha}\right] \leq \frac{\delta}{\log_2(2T)}.
\end{equation}
To complete this part of the proof, we union bound over 
$t \in \{2, 4, 8, \ldots, 2^{\lceil \log_2(T) \rceil}\}$. Then, for any 
particular $X_s$, there is some $t \leq 2s$ for which \eqref{eq:batch_exp_cvx_1} 
holds, and hence $X_s \leq \frac{16B^2kt\log(\log_2(2T)/\delta)}{\alpha}$.


We now apply Proposition~\ref{prop:adaptive} to the regularizers 
defined by \eqref{eq:batch_reg}, which yields
\begin{align*}
\sum_{s=1}^t &s\p{\p{w_s-u}^{\top}\partial f_s(w_s)} \\
&\leq \frac{\alpha}{4}\sum_{s=1}^t s\|w_s-u\|_2^2 + \frac{B^2}{2}\sum_{s=1}^t \frac{s^2k}{s(s+1)\alpha/4} \\
&\ \ \ \ + \lambda_t\Norm{u}_1 + \sum_{s=1}^t \p{\lambda_{s-1} - \lambda_s}\Norm{w_s}_1 \\
&\leq \frac{\alpha}{4}\sum_{s=1}^t s\|w_s-u\|_2^2 + \frac{2B^2kt}{\alpha} 
 + \lambda_t\Norm{u}_1 + \sum_{s=1}^t \p{\lambda_{s-1} - \lambda_s}\Norm{w_s}_1.
\end{align*}
Combining this inequality with \eqref{eq:batch_exp_cvx_0} 
yields the desired result.
\end{proof}

\begin{proof}[Proof of Theorem \ref{theo:l1_recipe_batch}]
With the given $\LI$ regularization schedule, we have
\begin{align*}
 \lambda_t &\Norm{w^*}_1 + \sum_{s = 1}^t \p{\lambda_{s - 1} - \lambda_s} \Norm{w_s}_1 \\
 &= \lambda\sum_{s = 1}^t \p{s^{3/2} - (s-1)^{3/2}} \p{\Norm{w^*}_1 - \Norm{w_s}_1}\\
  &\leq \frac{3}{2}\lambda \sum_{s = 1}^t \sqrt{s} \Norm{w^* - w_s}_1 \\
 &\leq \frac{3}{2} \lambda\sqrt{k} \sum_{s = 1}^t \sqrt{s} \Norm{w^* - w_s}_2 \\
 &\leq \frac{3}{2}\lambda \sqrt{kt\sum_{s=1}^t s\Norm{w_s-w^*}_2^2}.
\end{align*}
Meanwhile, by invoking Lemma~\ref{lemm:variance-bound-batch}, we have that, 
with probability at least $1-\delta$,
\begin{align*}
\sum_{s=1}^t &s\Norm{w_s-w^*}_2^2 \leq \frac{2}{\alpha} \left(\sum_{s=1}^t s\p{f_s\p{w_s} - f_s\p{w^*}} \right. \\
& \phantom{++}\left. + \sqrt{kB^2t\log\p{\frac{\lg(2T^3)}{\delta}}\max\p{\frac{2kB^2}{\alpha^2},\frac{9}{4}\sum_{s=1}^t s\Norm{w_s-w^*}_2^2}}\right). \\
\end{align*}
Applying Lemma~\ref{lemm:self-bound}, we find that
\begin{align}
\label{eq:param_error}
\sum_{s=1}^t &s\Norm{w_s-w^*}_2^2 \\
\notag
&\leq \frac{4}{\alpha} \sum_{s=1}^t s\p{f_s\p{w_s} - f_s\p{w^*}} \\
\notag
 &\phantom{++} + \frac{kB^2}{\alpha^2}\max\p{\sqrt{8t\log\p{\frac{\lg(2T^3)}{\delta}}},9t\log\p{\frac{\lg(2T^3)}{\delta}}} \\
\notag
 &= \frac{4}{\alpha} \sum_{s=1}^t s\p{f_s\p{w_s} - f_s\p{w^*}} + \frac{9kB^2t}{\alpha^2}\log\p{\frac{\lg(2T^3)}{\delta}}.
\end{align}
Thus, with probability $1 - \delta$,
\begin{align*}
\lambda_T &\Norm{w^*}_1 + \sum_{t = 1}^T \p{\lambda_{t - 1} - \lambda_t} \Norm{w_t}_1 \\
&\leq \frac{3\lambda}{2} \sqrt{\frac{4kT}{\alpha} \sum_{t=1}^T t \, \p{f_t\p{w_t} - f_t\p{w^*}} + \frac{9k^2B^2T^2}{\alpha^2}\log\p{\frac{\lg(2T^3)}{\delta}}}.
\end{align*}
Combining this inequality with \eqref{eq:l1_recipe_init_batch} and Lemma \ref{lemm:self-bound} we obtain the first inequality in \eqref{eq:l1_recipe_goal_batch}. 
To get the second, we simply use the fact that
\begin{align*}
\frac{3B\lambda}{2}\sqrt{9\log(\lg(2T^3)/\delta)} &\leq 
 \lambda^2 + 6B^2\log(\lg(2T^3)/\delta)
\end{align*}
by the AM-GM inequality.
\end{proof}

\begin{proof}[Proof of Lemma~\ref{lemm:variance-bound-batch}]
As in the proof of Lemma~\ref{lemm:variance-bound}, we will invoke the 
version of the Azuma-Hoeffding inequality given in 
Lemma~\ref{lemm:better_azuma}. In particular, let 
$$ Z = \sum_{t=1}^T t\p{f_t\p{w_t} - f_t\p{w^*}} $$
and take the filtration defined by $f_{1:T}$. Then, using the notation of
Lemma~\ref{lemm:better_azuma}, we have
$$ \Delta_t = t\p{f_t\p{w_t} - \law\p{w_t}} - t\p{f_t\p{w^*} - \law\p{w^*}} $$
by assumption (1).
Meanwhile, by the Lipschitz assumption, we have that 
$\Abs{f_t\p{w_t} - f_t\p{w^*}} \leq L\Norm{w_t - w^*}_2$, hence 
$M_t = tL\Norm{w_t-w^*}_2$ and also $\Norm{w_t-w^*}_2 \leq \frac{L}{\alpha}$. 
If we take $\sigma_2 = \frac{L^2}{\alpha}$, $\sigma_1 = T$, then the result follows directly from 
Lemma~\ref{lemm:better_azuma} and the bound
$$\sum_{s=1}^t M_s^2 = \sum_{s=1}^t s^2L^2\Norm{w_t-w^*}_2^2 \leq tL^2\sum_{s=1}^t s\Norm{w_t-w^*}_2^2. $$
\end{proof}

\subsection{Proofs for Section~\ref{sec:irrep}}

\begin{proof}[Proof of Lemma \ref{lem:irrep_lambda}]
At a high level, our proof is based on the following inductive argument: 
if $|\theta_{s,j}| \leq \lambda \cdot s^{3/2}$ for all $s \leq t$, then 
$\supp(w_s) \subseteq S$ for all $s \leq t$, and we thus have small excess risk, 
which will allow us to then show that $|\theta_{t+1,j}| \leq \lambda \cdot (t+1)^{3/2}$.

We start by showing that, for all $t \leq T$ and $j \not \in S$, if $\supp(w_s) \subseteq S$ 
for all $s \leq t$ then we have:
\begin{align}
\label{eq:irrep-1}
&\left| \sum_{s=1}^t s\, \partial f_s(w^*)_j\right| \leq \frac{3Bt^{3/2}}{2}\sqrt{\log\p{\frac{2d\lg(2T^3)}{\delta}}} \\
\label{eq:irrep-2}
&\left| \sum_{s=1}^t  s\p{\partial f_s(w_s)-\partial \law(w_s)-\partial f_s(w^*)}_j\right| \leq 3Bt^{3/2}\sqrt{\log\p{\frac{2d\lg(2T^3)}{\delta}}} \\
\label{eq:irrep-3}
&\left| \sum_{s=1}^t s\, \partial \law(w_s)_j\right| \leq \frac{\rho\alpha (t+1)}{\sqrt{2k}}\sqrt{\sum_{s=1}^t s\Norm{w_s-w^*}_2^2} 
\end{align}
Inequalities \eqref{eq:irrep-1} and \eqref{eq:irrep-2} 
each hold with probability 
$1-\delta$ while \eqref{eq:irrep-3} holds deterministically. 
Note that these inequalities immediately provide a bound on $|\theta_{t,j}|$, since 
\begin{align*}
\theta_{t} &= \sum_{s=1}^t s\, \partial f_s(w_s) \\
 &= \sum_{s=1}^t s\p{\partial f_s(w^*) + \p{\partial f_s(w_s) - \partial \law(w_s) - \partial f_s(w^*)} + \partial \law(w_s)}.
\end{align*}
To prove the claimed inequalities, note that each term on the 
left-hand-side of both \eqref{eq:irrep-1} and \eqref{eq:irrep-2} is zero-mean, 
so these inequalities both follow directly from applying 
Lemma~\ref{lemm:better_azuma} with $M_t = B$ and $M_t = 2B$, respectively (here we 
use the fact that $|\partial f_s(w)_j| \leq B$). The interesting inequality 
is \eqref{eq:irrep-3}, which holds by the following:
\begin{align*}
\left|\sum_{s=1}^t s\, \partial \law(w_s)_j\right| &\leq \sum_{s=1}^t s\left|\partial \law(w_s)_j\right| \\
 &= \sum_{s=1}^t s\left|\EE{(y-w_s^{\top}x)x_j}\right| \\
 &= \sum_{s=1}^t s\left|\EE{((w^*-w_s)^{\top}x)x_j}\right| \\
 &= \sum_{s=1}^t s\left|\Cov{x_j,x^{\top}(w^*-w_s)}\right|.
 \end{align*}
 Continuing, we find that
 \begin{align*}
 \left|\sum_{s=1}^t s\, \partial \law(w_s)_j\right|
 &\leq \frac{\rho\alpha}{\sqrt{k}}\sum_{s=1}^t s\Norm{w^*-w_s}_2 \; \; \; \text{(by Assumption~\ref{assu:irrep_noise})} \\
 &\leq \frac{\rho\alpha}{\sqrt{k}}\sqrt{\p{\sum_{s=1}^t s}\p{\sum_{s=1}^t s\Norm{w^*-w_s}_2^2}} \\
 &\leq \frac{\rho\alpha(t+1)}{\sqrt{2k}}\sqrt{\sum_{s=1}^t s\Norm{w^*-w_s}_2^2}.
\end{align*}
Now, from the comments at the top of the proof of Theorem~\ref{theo:main_irrep}, we also have the 
following bound for each $t \leq T$, provided that $\supp(w_s) \subseteq S$ for all 
$s \leq t$: 
\begin{align}
\label{eq:irrep-4}
&\sum_{s=1}^t s\Norm{w_s-w^*}_2^2 \leq \frac{kt}{\alpha^2}\p{177B^2\log\p{\frac{\lg(2T^3)}{\delta}} + 40\lambda^2}.
\end{align}
This bound holds with probability $1-2\delta$ and so all of the bounds together hold with 
probability $1-4\delta$. Arguing by induction, we need to show that if 
$\supp(w_s) \subseteq S$ for all $s \leq t$, then $\supp(w_{t+1}) \subseteq S$ as well. By 
the inductive hypothesis, we know by inequalities (\ref{eq:irrep-1}-\ref{eq:irrep-4}) that 
\begin{align*}
\sup_{j \not\in S} |\theta_{t+1,j}| &\leq \frac{9Bt^{3/2}}{2}\sqrt{\log\p{\frac{2d\lg(2T^3)}{\delta}}} \\
&\ \ \ \ \ \ + \rho (t+1)t^{1/2}\sqrt{\frac{177B^2}{2}\log\p{\frac{\lg(2T^3)}{\delta}} + 20\lambda^2}.
\end{align*}
Remember that we need $\lambda (t+1)^{3/2} \geq \sup_{j \not\in S} |\theta_{t+1,j}|$. Therefore, 
using the inequality $(a+b)^2 \leq 6a^2+1.2b^2$, it suffices to take 
$$\lambda^2 \geq 228B^2\log\p{\frac{2d\lg(2T^3)}{\delta}} + 24\rho^2\lambda^2. $$
We therefore see that we can take any $\lambda$ with 
$\lambda^2 \geq \frac{228B^2\log(2d\lg(2T^3)/\delta)}{1-24\rho^2}$, as was to be shown.
\end{proof}

\begin{proof}[Proof of Theorem \ref{theo:main_irrep}]
Suppose that $w_s \subseteq S$ for all $s \leq t$. Then, by 
Theorems~\ref{theo:sparse_expected_cvx_batch} and \ref{theo:l1_recipe_batch} and \eqref{eq:param_error}, we have 
with probability $1-2\delta$ that the following two inequalities hold for all $t$:
\begin{align}
\label{eq:with-lambda}
\sum_{s = 1}^t &s\p{f_s\p{w_s} - f_s\p{w^*}} \\
\notag
&\leq \frac{kt}{\alpha}\p{42B^2\log\p{\frac{\lg(2T^3)}{\delta}} + 10\lambda^2}, \text{ and }
\end{align}
\begin{align}
\label{eq:with-lambda2}
\sum_{s = 1}^t &s\|w_s-w^*\|_2^2 \\
\notag
&\leq \frac{4}{\alpha}\sum_{s=1}^t s\p{f_s(w_s)-f_s(w^*)} + \frac{9kB^2t\log(\log_2(2T^3)/\delta)}{\alpha^2} \\
\notag
 &\leq \frac{kt^2}{\alpha^2}\p{177B^2\log\p{\log_2(2T)/\delta} + 40\lambda^2}.
\end{align}
Thanks to Lemma \ref{lem:irrep_lambda}, we can verify that these relations in fact
hold for all $t \leq T$ with total probability $1-4\delta$,
provided that $\lambda$ satisfies \eqref{eq:irrep_lambda}.

To complete the proof, we use the online-to-batch conversion bound from 
Proposition~\ref{prop:batch}. With probability $1-5\delta$, we then have
\begin{align*}
\lefteqn{\law\p{\frac{2}{T(T+1)}\sum_{t=1}^T tw_t} - \law\p{w^*}} \\
 &\leq \frac{4k}{\alpha t}\p{42B^2\log\p{\frac{\lg(2T^3)}{\delta}} + 10\lambda^2} + \frac{9kB^2}{\alpha t}\log\p{\frac{\lg(2T^3)}{\delta}} \\
 &= \oop{\frac{k}{\alpha t}\p{B^2\log\log(T) + \lambda^2}}.
\end{align*}
Since we can take $\lambda^2$ to be $\oop{\frac{B^2\log(d\log(T))}{1-24\rho^2}}$, 
we can attain a bound of $\oop{\frac{kB^2\log(d\log(T))}{\alpha T(1-24\rho^2)}}$, 
which completes the theorem.
\end{proof}

\end{appendix}

\end{document}